\documentclass[a4paper]{amsart} 
\usepackage{amsmath,amsthm,amssymb,amsfonts,mathrsfs,color,hyperref, mathtools,crop, graphicx, enumitem, tikz, pgfplots}
\usepackage[]{geometry}
\usepackage{todonotes}
\usepackage{color}

\usepackage[OT2,OT1]{fontenc}
\newcommand\cyr{%
\renewcommand\rmdefault{wncyr}%
\renewcommand\sfdefault{wncyss}%
\renewcommand\encodingdefault{OT2}%
\normalfont
\selectfont}
\DeclareTextFontCommand{\textcyr}{\cyr}

\theoremstyle{plain}
\begingroup
\newtheorem{theorem}{Theorem}[section]
\newtheorem*{theorem*}{Theorem}
\newtheorem*{"theorem"}{``Theorem''}
\newtheorem{corollary}[theorem]{Corollary}
\newtheorem{proposition}[theorem]{Proposition}

\newtheorem{lemma}[theorem]{Lemma}
\endgroup

\theoremstyle{definition}
\begingroup

\endgroup

\theoremstyle{remark}
\begingroup
\newtheorem{remark}[theorem]{Remark}

\endgroup 

\numberwithin{equation}{section}
\newcommand{\N}{\mathbb N} 
\newcommand{\T}{\mathbb T} 
\newcommand{\Z}{\mathbb Z} 
\newcommand{\R}{\mathbb R} 

\newcommand{\dist}{{\rm dist}}

\newcommand{\wto}{\rightharpoonup}
\newcommand{\ol}{\overline}
\newcommand{\E}{{\mathcal E}}

\DeclareMathOperator{\A}{A}

\renewcommand{\L}{{\mathcal L}}

\newcommand{\F}{{\mathcal F}}

\newcommand{\LRa} {\Leftrightarrow}
\newcommand{\Ra} {\Rightarrow}

\renewcommand{\H}{{\mathcal H}}

\let\ul = \underline
\let\wt = \widetilde
\newcommand{\sign}{\mathrm{sign}}

\newcommand{\eps}{\varepsilon}

\newcommand{\dx}{\,\mathrm{d}x}
\renewcommand{\d}{\,\mathrm{d}}
 \newcommand{\dy}{\,\mathrm{d}y}

 \newcommand{\dt}{\,\mathrm{d}t}
 \newcommand{\dr}{\,\mathrm{d}r}
\newcommand{\ubar}{\bar u}
\newcommand{\w}{\overline w}
\newcommand{\Per}{\mathrm{Per}}
 \newenvironment{pde}{\left\{\begin{array}{rll} } {\end{array}\right.}

\begin{document}

\title[The Effect of Forest Dislocations]{The Effect of Forest Dislocations on the Evolution of a Phase-Field Model for Plastic Slip}

\author{Patrick W.~Dondl}
\address{Patrick W.~Dondl\\Albert-Ludwigs-Universit\"at Freiburg\\Mathematisches Institut\\Eckerstraße 1\\79104 Freiburg, Germany}
\email{patrick.dondl@mathematik.uni-freiburg.de}

\author{Matthias W. Kurzke}
\address{Matthias W. Kurzke\\Mathematical Sciences\\University Park\\Nottingham
NG7\,2RD, UK}
\email{matthias.kurzke@nottingham.ac.uk}

\author{Stephan Wojtowytsch}
\address{Stephan Wojtowytsch\\Department of Mathematical Sciences\\Durham University\\Durham DH1\,1PT, United Kingdom}
\email{s.j.wojtowytsch@durham.ac.uk}

\date{\today}


\begin{abstract}
We consider the gradient flow evolution of a phase-field model for crystal dislocations in a single slip system in the presence of forest dislocations. The model consists of a Peierls-Nabarro type energy penalizing non-integer slip and elastic stress. Forest dislocations are introduced as a perforation of the domain by small disks where slip is prohibited. The $\Gamma$-limit of this energy was deduced by Garroni and M{\"u}ller (2005 and 2006). Our main result shows that the gradient flows of these $\Gamma$-convergent energy functionals do not approach the gradient flow of the limiting energy. Indeed, the gradient flow dynamics remains a physically reasonable model in the case of non-monotone loading. Our proofs rely on the construction of explicit sub- and super-solutions to a fractional Allen-Cahn equation on a flat torus or in the plane, with Dirichlet data on a union of small discs. The presence of these obstacles leads to an additional friction in the viscous evolution which appears as a stored energy in the $\Gamma$-limit, but it does not act as a driving force. Extensions to related models with soft pinning and non-viscous evolutions are also discussed. In terms of physics, our results explain how in this phase field model the presence of forest dislocations still allows for plastic as opposed to only elastic deformation.
\end{abstract}

\maketitle

\setcounter{tocdepth}{1}

\tableofcontents

\section{Introduction}
It is well-known that $\Gamma$-convergence of functionals is a $C^0$-type convergence that does not imply convergence of the related dynamics. For example, the `wiggly' potentials
\[
f_\eps:[-1,1]\to \R, \qquad f_\eps(x) = x^2 + 2\eps\,\sin(x^2/\eps)
\]
converge uniformly, hence also in the sense of $\Gamma$-convergence, to the limit $f(x) = x^2$ as $\eps\to 0$ while solutions to the gradient flows of $f_\eps$ never move more than $\pm \pi \eps$ from their initial datum into a local minimum and thus do not resemble the gradient flow of the $\Gamma$-limit at all. 

On the other hand, there are well known conditions under which the gradient flows of $\Gamma$-convergent functionals on Hilbert spaces \cite{sandier2004gamma} and metric spaces \cite{serfaty2011gamma} approach the gradient flow of a the limiting energy in a suitable sense. In applications, it is not always obvious whether functionals belong to the `wiggly' or the convergent `Sandier-Serfaty'-class.

Here, we consider the effect of forest dislocations on the propagation of slip in Peierls-Nabarro-type models following \cite{MR1935021}. In~\cite{MR2178227, MR2231783} it was shown that the corresponding non-local Modica-Mortola type energy functional augmented with the condition that the phase field $u_\eps$ vanishes at certain small obstacles $\Gamma$-converges to a functional given by the sum of a perimeter and a bulk energy. Essentially, the articles above show (in higher generality) that the  energies 
\begin{equation}\label{eq: energy}
\E_\eps(u_\eps) =  \frac1{|\log\eps|}\left(\,[u_\eps]^2_{1/2} + \int_{\T^2} \frac1\eps W(u_\eps)\dx\right)
\end{equation}
converge to a functional
\begin{equation}\label{eq: limit energy}
\E(u) = \Per(\{u=1\}) + \Lambda\,\alpha\,\H^2(\{u=1\}), \qquad u\in BV(\T^2, \{0,1\})
\end{equation}
in the sense of $\Gamma$-convergence with respect to the strong $L^2$-topology when restricted to the spaces
\begin{equation}\label{eq ansatz space}
\E_\eps:X_\eps \to \R, \qquad X_\eps:= \{u_\eps \in H^{1/2}(\T^2)\:|\:u_\eps\equiv 0\text{ on }B_\eps(x_{i,\eps})\text{ for } 1\leq i\leq N_\eps\}.
\end{equation}
The obstacles $x_{i,\eps}$ have to satisfy certain distribution assumptions and $\frac\eps{|\log\eps|} \sum_{i=1}^{N_\eps}\delta_{x_{i,\eps}} \wto \Lambda\,\H^2$,
where $\H^k$ denotes the $k$-dimensional Hausdorff measure and $W$ is a non-negative smooth multi-well potential vanishing quadratically at the integers. The constant $\alpha$ is determined through the solution of a cell-problem. Dislocations in this model are given by $\Z+1/2$-level sets of the slip $u$ (see Figure \ref{figure one}).

The focus of this article is the evolution that arises as the limit of the gradient flows of the energies $\E_\eps$. In technical terms, we are interested in the behaviour of solutions to the evolution equation
\begin{equation}\label{eq: evolution}
c_\eps\,(\eps \,\partial_tu_{\eps}) = \frac1{|\log\eps|}\left(\A u_{\eps} - \frac1\eps\,W'(u_{\eps})\right), \qquad u_\eps(0, \cdot) = u_\eps^0, \qquad u_\eps(t,\cdot) \in X_\eps\quad\forall\ t>0
\end{equation}
as $\eps\to 0$, where $\A \coloneqq - (-\Delta)^{1/2}$ is the fractional Laplacian or order $s=1/2$. The case $c_\eps\equiv 1$ corresponds to the time-normalised gradient flow dynamics of \eqref{eq: energy} under the pinning constraint \eqref{eq ansatz space}. Depending on the exact problem, we identify the scaling regime $c_\eps$ in which the evolution approaches non-trivial limiting dynamics and give results on the limits of solutions of \eqref{eq: evolution} for suitable initial conditions. 

We show that this problem belongs to the `wiggly' world, i.e., that the gradient flows of  $\E_\eps$ do not approach the dynamics of the limiting problem. The idea behind this is that the non-locality in the energy is too weak to summon a driving force on an otherwise unloaded flat dislocation from the pinning constraint on a relevant time scale. On the other hand, if an external force (or a curvature term) acts to expand the $\{u_\eps\approx 1\}$-phase, we do see a resistance from the energy barrier. Thus the perforation of the domain induces a friction which only resists other forces but does not initiate movement.

Three different terms appear in the dynamics on different time-scales: The curvature driven evolution stemming from the diffuse line-energy which acts on the gradient flow scale, a non-local interaction between interfaces (kink/kink repulsion and kink/anti-kink attraction) stemming from the next order $\Gamma$-limit which is $|\log\eps|^{-1}$-small with respect to the curvature flow, and the diffuse bulk term, which acts as a driving force, but only on a time-scale which is roughly $\eps^{{1/2}}$-slow with respect to the other terms, although it can act against other forces on the fast scale.

In this article, we mostly focus on the situation of infinite parallel straight interfaces to be able to neglect the curvature-driven evolution. We  construct explicit sub- and super-solutions to estimate the speed of motion of an interface. At some non-straight interfaces, we can obtain bounds by using sub-solutions as barriers to show non-expansive behaviour (i.e., $u$ non-decreasing) and energy methods to show non-shrinking of an initial condition (i.e., $u$ non-increasing). 

Physically, our results provide a justification why the phase-field model is valid beyond the applicability of the $\Gamma$-limit, where the bulk term stemming from the forest dislocations induces a dislocation evolution to return to the undeformed state $u=0$ at macroscopic velocities.

The article is structured as follows. In Section \ref{section heuristics}, we explain the mathematical setting and the heuristic reasoning behind our results as well as a brief statement of our main theorems. In Section \ref{section 1d} we construct sub-solutions and apply them to a one-dimensional analogue of our problem in order to obtain results in this simpler setting. In Section \ref{section 2d}, we state the main results in more precise and general terms and show how the one-dimensional proofs can be adapted to yield the full results. Section \ref{section stuff} is devoted to the discussion of different related models, in particular non-viscous evolution. We conclude the article with a brief summary and some open problems. In an appendix, we briefly discuss parabolic equations with fractional differential operators on bounded domains.

\section{Background and Heuristics} \label{section heuristics}

\subsection{The Energy Limit}

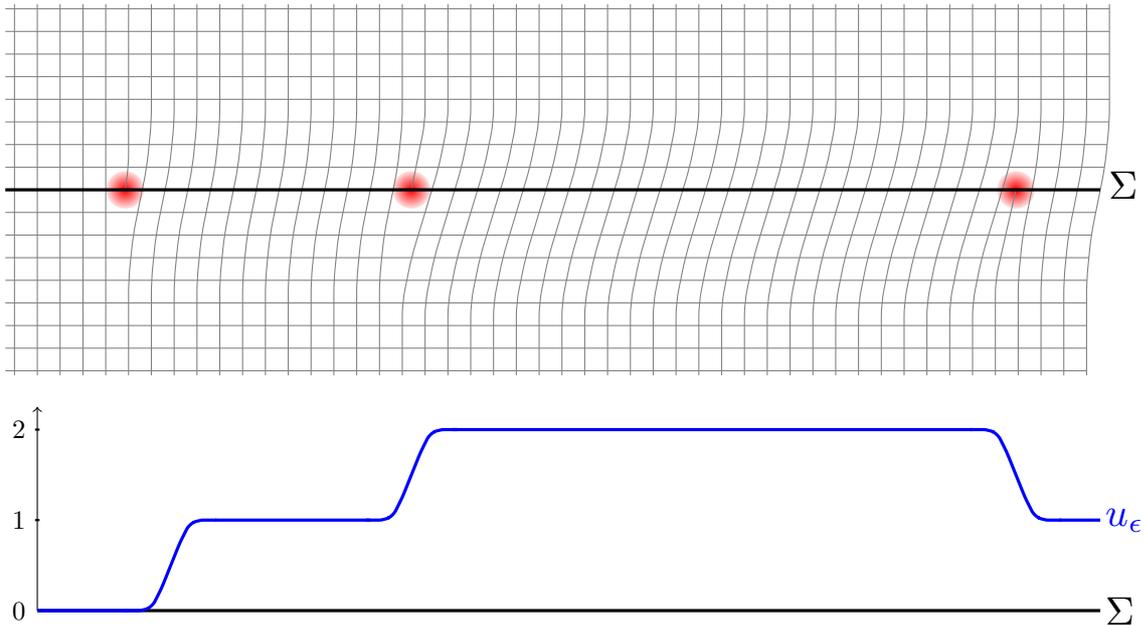
\begin{figure}
\begin{center}
\begin{tikzpicture}[scale = 0.6]

\definecolor{test}{gray}{0.5}
\definecolor{darkgreen}{RGB}{0,102,51}

\fill[ inner color=red, outer color=white] (1.925,4) circle (0.45);
\draw[white, very thick]  (1.925,4) circle (0.45);

\fill[ inner color=red, outer color=white] (8.2, 4) circle (0.45);
\draw[white, very thick]  (8.2, 4) circle (0.45);

\fill[ inner color=red, outer color=white] (21.45, 4) circle (0.45);
\draw[white, very thick]  (21.45, 4) circle (0.45);

\foreach\n in{0, ..., 16}
\draw[rounded corners = 15pt, test] (-0.7, \n/2) -- (23.7, \n/2);

\fill[white, rounded corners = 5pt] (46/2, -0.1) -- ++ (0,2.1) -- ++ (0.1,1) -- ++ (0.3,1.5) -- ++ (0.1,1) -- ++ (0,2.8)-- ++ (1,0)-- ++(0,-8.7);

\foreach\n in{-1, ..., 3}
\draw[rounded corners = 15pt, test] (\n/2, -0.1) -- (\n/2, 8.1);

\draw[rounded corners = 5pt, test] (1.925, 4) -- ++ (0.075,0.5) -- ++ (0,3.6);

\foreach\n in{4, ..., 15}
\draw[rounded corners = 5pt, test] (\n/2, -0.1) -- ++ (0,2.1) -- ++ (0.1,1) -- ++ (0.3,1.5) -- ++ (0.1,1) -- ++ (0,2.6);

\draw[rounded corners = 5pt, test] (8.2,4) -- ++ (0.3,1.5) -- ++ (0,2.6);

\foreach\n in {16, ..., 41}
\draw[rounded corners = 5pt, test] (\n/2, -0.1) -- ++ (0,1.6) -- ++ (0.2,1) -- ++ (0.6,2) -- ++ (0.2,1) -- ++ (0,2.6);

\draw[rounded corners = 5pt, test] (21, -0.1) -- ++ (0,2.1) -- ++ (0.45,2);

\foreach\n in {43, ..., 46 }
\draw[rounded corners = 5pt, test] (\n/2, -0.1) -- ++ (0,2.1) -- ++ (0.1,1) -- ++ (0.3,1.5) -- ++ (0.1,1) -- ++ (0,2.6);

\draw[very thick](-0.7, 4) -- (23.3, 4);
\node at (23.8, 4.1) [scale=1.5] {$\Sigma$};

\begin{scope}[yshift=-5.3cm]

\draw[ ->](0,0) -- ++(0,4.5);
\draw[ very thick](0,0) -- ++ (23.3, 0);

\node at (-0.4, 0) {0};
\node at (-0.4, 2) {1}; \draw [thick](-0.05, 2) -- ++ (0.1, 0);
\node at (-0.4, 4) {2};\draw [thick](-0.05, 4) -- ++ (0.1, 0);

\node at (23.1, 0)[,right, scale = 1.5] {$\Sigma$};

\node at (23.1, 2)[right,scale = 1.5, blue]{$u_\epsilon$};

\begin{scope}[xshift=1.925cm]
\draw[very thick, blue, rounded corners = 2pt] (0,0) -- ++ (0.4, 0) -- ++ (0.2,0.1) -- ++(0.2, 0.4)--++  (0.4, 1) -- ++(0.2,0.4) -- ++ (0.2,0.1) -- ++(0.4,0);
\end{scope}
\begin{scope}[xshift=7.2cm, yshift=2cm]
\draw[very thick, blue, rounded corners = 2pt] (0,0) -- ++ (0.4, 0) -- ++ (0.2,0.1) -- ++(0.2, 0.4)--++  (0.4, 1) -- ++(0.2,0.4) -- ++ (0.2,0.1) -- ++(0.4,0);
\end{scope}
\begin{scope}[xshift=22.45cm, yshift=2cm, xscale=-1]
\draw[very thick, blue, rounded corners = 2pt] (0,0) -- ++ (0.4, 0) -- ++ (0.2,0.1) -- ++(0.2, 0.4)--++  (0.4, 1) -- ++(0.2,0.4) -- ++ (0.2,0.1) -- ++(0.4,0);
\end{scope}

\draw[very thick, blue](0,0)--(2,0);
\draw[very thick, blue] (3.9, 2) -- (7.3,2);
\draw[very thick, blue] (9.1, 4) -- (20.5,4);
\draw[very thick, blue] (22.4,2) -- (23.3,2);

\end{scope}

\end{tikzpicture}
\end{center}
\caption{\label{figure one}The phase-field $u_\eps$ counts the number of half-planes wedged into a crystal grid. Dislocations are the level sets of $\Z+1/2$, i.e, the interfaces between the phases. The dislocations we consider all lie in the same plane $\Sigma$ of the crystal grid and their Burgers vectors $\vec b$ are integer multiples of a single vector $\vec{b_0}$. }
\end{figure}

The energies $\E_\eps$ are obtained as a model for crystal dislocations in \cite{MR1935021} and motivated in their current form in \cite{MR2178227,MR2231783}. As the characteristic length scale $\eps$ of crystal grids is typically very small compared to the behaviour of a crystal on the length scale we are interested in, it is desirable to have a simpler continuum limit $\eps\to 0$ available. This has been formalised by Garroni and M\"uller as follows.

\begin{theorem}\cite{MR2231783}\label{theorem garroni mueller}
Let $x_{i,\eps}\in \T^2$ be points such that $1\leq i\leq N_\eps$ with $\frac{\eps}{|\log\eps|}N_\eps\to \Lambda$ satisfying the following assumptions:

\begin{enumerate}
\item {\rm (equi-distributed)} For $r_\eps\sim N_\eps^{-1/2}$ there exist constants $c,C>0$ such that 
\[
c\,r_\eps^2\,N_\eps \leq N_\eps(Q_\eps)\leq C\,r_\eps^2\,N_\eps
\]
where $N_\eps(Q_\eps)$ is the number of obstacles in $Q_\eps$ and $Q_\eps$ is a square of side length $r_\eps$.

\item {\rm (well-separated)} There exists $\beta<1$ independent of $\eps>0$ such that $d(x_{i,\eps},x_{j,\eps})>6\,\eps^\beta$ for all $1\leq i\neq j\leq N_\eps$.

\item {\rm (finite capacity density)} The obstacles approach a multiple of the Lebesgue measure through $\frac{\eps}{|\log\eps|}\sum_{i=1}^{N_\eps}\delta_{x_i}\wto \Lambda\,\L^2$ for $\Lambda\in (0,\infty)$.

\end{enumerate}

Take the space
\[
X_\eps:= \{u_\eps \in H^{1/2}(\T^2)\:|\:u_\eps\equiv 0\text{ on }B_\eps(x_{i,\eps})\text{ for } 1\leq i\leq N_\eps\}
\]
and the energy functional
\[
\E_\eps:X_\eps\to\R, \qquad \E_\eps(u_\eps) =  \frac1{|\log\eps|}\left(\,[u_\eps]^2_{1/2} + \int_{\T^2} \frac1\eps W(u_\eps)\dx\right)
\]
where $W$ is a periodic multi well potential satisfying $W\geq c\,\dist^2(\cdot,\Z)$ for some $c>0$. Then
\[
\left[\Gamma(L^2)-\lim_{\eps\to 0} \E_\eps\right](u) = \int_{\T^2} \alpha(u)\dx + 4\int_{J_u}[u]\d\H^{1}
\]
where $u\in BV(\T^2,\Z)$, $[u] = u^+-u^-$ denotes the jump of $u$ on the jump set $J_u$ and $\alpha(z)$ is determined as the solution of the cell problem
\begin{equation}\label{eq cell problem}
\alpha(z) = \inf\left\{ \frac12\,[w]_{1/2,\R^2}^2 + \int_{\R^2}W(w)\dx \:\bigg|\: w-z\in H^{1/2}(\R^2), \:w\equiv 0 \text{ on }B_1(0)\right\}.
\end{equation}
\end{theorem}

In \cite{MR2178227, MR2231783} the precise statement is given also for anisotropic kernels, different scalings of the number of obstacles, different obstacle sizes proportional to $\eps$, and finite strength pinning. Furthermore, pre-compactness of finite energy sequences is established.

Let us briefly comment on this result. The $\Gamma$-limit is essentially the sum of two terms, the perimeter functional which occurs as the limit of the unconstrained non-local Modica-Mortola functional (see \cite{MR1657316,MR2948285} for double-well potentials and \cite{MR2231783,KurzkeBV,KurzkeGF} for periodic potentials), and the bulk term which stems from the pinning constraint (see \cite{MR0601059, MR1907765,MR1493040} for the local case). In the critical scaling $N_\eps \sim \frac{|\log\eps|}\eps$ both terms appear on the same order. 

\begin{remark}
In one dimension, the critical scaling is $N_\eps\sim |\log\eps|$. The difference arises due to the different scaling of the $H^{1/2}$-semi-norm in different dimensions.
\end{remark}

\subsection{Viscous Evolution}\label{section heuristic}

In this article, we  compare the gradient-flow dynamics associated to the functionals $\E_\eps$ with those of the continuum limit. If we assume that both halves of the crystal relax on a timescale much faster than the motion of dislocations, we can describe the dynamics by a quasi-static evolution, i.e.\ only the jump set along the slip plane needs to be evolved according to the gradient flow of the energy $\E_\eps$ and the distortion field in upper and lower half space approaches the associated energy minimum instantaneously.

According to \cite{imbert2009phase}, solutions to the associated evolution equation of $\E_\eps$ {\em without} the pinning constraint
\[
\eps\,u_t = \frac1{|\log\eps|}\left(\A u-\frac1\eps\,W'(u)\right)
\]
converge to level set mean curvature flow. 

\begin{remark}
The $\eps$ in front of the time derivative is the correct time scaling for a phase field gradient flow since the interface moves with speed $O(1)$ if the time derivative is $O(1/\eps)$.
\end{remark}

In one dimension, the perimeter functional has no interesting dynamics, so the behaviour of the evolution equation (without obstacles) should be governed by the next order $\Gamma$-limit. At a simple step function $\chi_{[r_1,r_2]}$ on the real line we can modify arguments from \cite{KurzkeBV, KurzkeGF} for closely related energies to see that
\begin{equation}\label{eq first order gamma}
{\Gamma(L^2)-\lim_{\eps\to 0} |\log\eps|\cdot \left(\E_\eps - 2\right) = c_0\,\log|r_2-r_1|}+c_1
\end{equation}
where $\E_\eps$ is given by the same formula as above, but in dimension one and on a space without pinning constraint. Here $c_0>0$ is a constant depending on the potential $W$.
 In particular, the next order term in the $\Gamma$-expansion vanishes only logarithmically in $\eps$ rather than exponentially fast as in the classical local functional. Using non-variational techniques, Gonzalez and Monneau  showed  in \cite{gonzalez2010slow}that in one dimension (or at straight parallel interfaces), we still expect to see attraction of interfaces on the slower timescale
\[
\frac1{|\log\eps|}\left(\eps\,u_t\right) = \frac1{|\log\eps|}\left(\A u- \frac1\eps\,W'(u) \right).
\]
This surprisingly fast motion contrasts with the (local) Allen-Cahn equation in one dimension
\[
\eps\,u_t = \eps\,\Delta u - \frac1\eps\,W'(u),
\]
which becomes exponentially slow in $\eps$ \cite{carr1989metastable}. The stronger attraction here stems from the non-locality of the half Laplacian as compared to the full Laplace operator, stemming from the next order $\Gamma$-limit \eqref{eq first order gamma} at a simple step function. The heavy tails of the singular kernel force slower decay of optimal interfaces for the fractional Allen-Cahn equation, which translates into stronger attraction (see Section \ref{section interface}).

Now consider the effect of pinning, just in one dimension. Heuristically, we simply take a function $u$ with one or two interfaces and pinned obstacles on points $d_\eps\Z$ such that the interfaces are $O(d_\eps)$ away from the nearest obstacle, $0$ outside and $1$ in between the obstacles. The obstacle at $md_\eps$ contributes an amount roughly proportional to
\[
\frac1{|\log\eps|}\int_{m\,d_\eps-\eps}^{m\,d_\eps +\eps}\frac{1}{|x|^2}\dx \approx \frac{\eps}{|\log\eps|\,m^2\,d_\eps^2}
\]
to the attractive force in the $1/2$-Laplacian on an interface at $x=0$, using the representation
\begin{align}\nonumber
\Delta^{1/2}u(x) &= P.V.\int_{\R^n} \frac{u(y) - u(x)}{|x-y|^{n+1}}\dy\\
	&= \int_{B_\rho(x)} \frac{u(y) - u(x) - \langle \nabla u(x), y-x\rangle}{|y-x|^{n+1}}\dy + \int_{\R^n\setminus B_\rho(x)} \frac{u(y) - u(x)}{|x-y|^{n+1}}\dy\label{eq fracoperator form}
\end{align}
 as a singular integral operator. The expression $P.V.\int$ denotes that the integral needs to be understood in the principal value sense $P.V.\int_{\R^n} = \lim_{\eps\to 0} \int_{\R^n\setminus B_\eps(0)}$. This interpretation will be implied in the following. The integrals in the second expression exist and use the symmetry of the integral kernel and the antisymmetry of the linear term for cancellation effects and $\rho\in(0,\infty)$ can be chosen freely. This form will be frequently used for estimates in the following. Note that our normalisation of the fractional Laplacian (and the $H^{1/2}$-semi-norm) differ from the usual one by a dimension-dependent constant.
 
  We can sum over $m\in\Z$ and obtain a term proportional to $\eps/(d_\eps^2\,|\log\eps|)$, which is much smaller than the attractive force between interfaces for $d_\eps\gg \sqrt{\eps/|\log\eps|}$. Seeing that the interesting amount of obstacles in one dimension would be $N_\eps\sim |\log\eps|$ on a periodic interval, the natural distance between obstacles scales as $d_\eps\sim1/|\log\eps|$. We are lead to the conjecture that the obstacles' contribution to the the contracting force vanishes in the limit $\eps\to 0$, and that the dynamics are independent of the presence of obstacles in this scenario.

In one dimension, the pinning is expected to have an effect on the evolution in one dimension if $d_\eps \sim \sqrt{\eps/|\log\eps|}$, which is the natural length scale in two dimensions (since the natural scaling for the number of obstacles is $N_\eps \sim \frac{|\log\eps|}\eps$). We would still expect two-dimensional solutions to become slow in this scaling since the one-dimensional case corresponds to solutions constant in one direction or the effect of pinning along whole lines, not just on circles.

A two dimensional version of the argument above gives the contribution
\[
\frac{1}{|\log\eps|}\int_{B_\eps(id_\eps, jd_\eps)} \frac1{|x|^3}\dx \approx \frac{\eps^2}{|\log\eps|\,d_\eps^3}\,\frac{1}{(i^2+j^2)^{3/2}}
\]
for a single obstacle and thus the scaling proportional to $\eps^2/(d_\eps^3\,|\log\eps|)$ for the contribution of the obstacles to the driving force. Again, inserting $d_\eps = 1/\sqrt{N_\eps}= \sqrt{\eps/|\log\eps|}$, we see that this force should be $O(\eps^{1/2}|\log\eps|^{1/2})$ which is negligible compared to the attraction between interfaces, let alone curvature. 

Technically, the relevant consideration is whether this back-of-the-envelope calculation gives the right scaling or whether the pinning induces further non-local effects. In particular, we need to investigate how quickly minimisers of the cell-problem \eqref{eq cell problem} approach $z\in\Z$ at infinity. 

Another interesting question is how pinning interacts with other terms. Namely, when external forces, the attraction of interfaces, or curvature terms are driving an interface to expand the phase $\{u=0\}$, a moving interface must create new obstacles during the movement (for example by Orowan loops). This would lead to an increase in the bulk energy term which may dominate the potential energy gain. In this case, the presence of obstacles prevents motion. 

These heuristic considerations suggest that the forest dislocations do not act as a driving force on the relevant time-scale, but may act against other driving forces to prevent motion. In this sense, it is more appropriate to think of the obstacles as creating a friction term in the dynamic case rather than a stored energy as it appears in the $\Gamma$-limit. Studying the gradient flow of the present phase field model thus provides insight into the treatment of stored energy hardening terms in macroscopic models for plastic evolution, in particular how to include a Bauschinger effect.

\subsection{Main Results and Idea of Proof}\label{section main}

We will always assume that $W\in C^{2,\alpha}(\R)$ for some $\alpha>0$, that $W$ is $1$-periodic and satisfies $W\geq c\,\dist^2(\cdot,\Z)$ for some $c>0$ and $W(0)=0$. Note that the conditions together imply that $W''(0)>0$. Additional conditions will be placed on $W$ in Sections \ref{section interface} and \ref{section corrector} to ensure the right behaviour of the second derivatives of the optimal transition profile between two neighbouring potential wells in one dimension and of a corrector function for moving interfaces. The prototype of an admissible potential is $W(z) = \sin(\pi z)+1$. Furthermore, we make the following assumptions on the distribution of obstacles $x_{i,\eps}$:

\begin{enumerate}
\item the assumptions of Theorem \ref{theorem garroni mueller} hold and additionally
\item the obstacles are arranged as perturbations of a square grid. 
\end{enumerate}

The second condition is stated in a precise fashion in Theorem \ref{theorem 2d torus}. Admissible configurations are a perfect square grid with length scale $d_\eps \sim N_\eps^{-1/2}$, small perturbations of the grid on a small fraction of this length scale, or a square grid on a slightly smaller length scale with vacancies and potentially multiple points $x_{i,\eps}$ close to a single node of the grid. A truly random arrangement of $x_{i,\eps}$ as identically and uniformly distributed points on $\T^2$ is admissible neither for our results nor in Theorem \ref{theorem garroni mueller}.

\begin{theorem*}
Under the assumptions above, we prove the following:
\begin{enumerate}
\item The gradient-flows of $\E_\eps$ do not converge to the gradient flow of $\E$ in any time-scale, nor to pure mean curvature flow.

\item \label{theoremitem2} For a suitably aligned single straight interface on $\R^2$, a time-rescaling $c_\eps\leq \sqrt{\frac{\eps}{|\log\eps|}}$ is necessary to obtain a moving interface in the limit $\eps\to 0$. In the plane or on a flat torus, two suitably aligned and sufficiently close interfaces attract on a time-scale of $c_\eps = \frac{1}{|\log\eps|}$ independently of the presence of obstacles $x_{i,\eps}$.

\item If we apply an external force to increase the amount of slip $|u|$ or mean curvature flow would act in that way, the Garroni-M\"uller energy barrier has to be overcome and the presence of obstacles can prevent such motion.
\end{enumerate}
\end{theorem*}

\begin{remark}
Point \eqref{theoremitem2} in the Theorem can be interpreted in the sense that in the unrescaled time-scale, dislocations remain stationary after unloading in contrast to the evolution of the $\Gamma$-limit. This extends the validity of the phase-field model to non-monotone loading.
\end{remark}

\begin{remark}
Informally speaking, the essence of our results can be stated as follows. In the gradient flow time scaling and in the presence of forest dislocations, straight parallel dislocation lines are stationary. If an exterior force is applied in the direction of increasing the amount of slip, the dislocations remain stationary until a certain threshold is reached, while they offer neither resistance nor help to an exterior force which acts in the direction of decreasing the amount of slip, see Figure \ref{figure effective}. In particular, the results derived here are consistent with the mechanical Bauschinger effect observed in~\cite{Demir:10}, in the sense that reverting a plastic deformation is associated with a dramatic yield strength drop, but the reversal does of course not take place spontaneously.

To simplify the constructions, we have presented proofs for slip functions taking only the values in $[0,1]$, but extensions to positive slip are possible. For technical reasons, we focus on {\em signed} slip and interfaces which are aligned with the forest dislocations. By that we mean that if the forest dislocations are located on a square grid $d_\eps \cdot\Z^2$ and a straight interface in $\R^2$ meets the $x$-axis at an angle $\phi \in [0,2\pi)$, then we require $\tan(\phi)\in \mathbb Q$. We believe that also this restriction is of a purely technical nature. 

In one dimension, the second restriction does not appear and the results are sharp.
\end{remark}

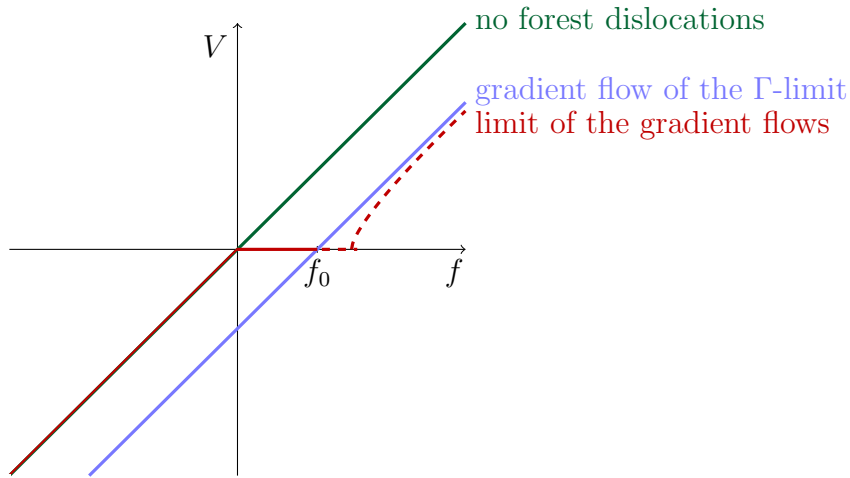
\begin{figure}
\begin{center}
\begin{tikzpicture}[scale=1.5]
\definecolor{darkgreen}{RGB}{0 100 50}
\definecolor{plotblue}{RGB}{120 120 255}
\definecolor{plotred}{RGB}{192 0 0}

\draw[->] (-2, 0 ) -- (2,0);
\draw[->] (0, -2) -- (0,2);

\node[left] at (0, 1.8){\Large $V$};
\node[below] at (1.9, 0) {\Large $f$};
\node[below] at (0.7, 0) {\Large $f_0$};
\draw[] (0.7, -0.03) -- ++ (0, 0.06);

\draw[very thick, darkgreen] (-2, -2) --++  (4,4);

\draw[very thick, plotblue] (-1.3, -2) -- (2, 1.3);

\draw[ thick, plotred] (-2,-1.995) -- ++ (2,2); 
\draw[very thick, plotred] (-0.004,0) -- (0.7, 0); 
\draw[very thick, plotred, dashed] (0.74,0) -- (1.05,0); 
\draw[very thick, plotred, dashed, domain = 0: 1, samples =100] plot ( {\x+1},  { 0.9 *\x + 1.3* sqrt(\x) /(\x+3) } ); 

\fill[white] (-2.01, -2.01) -- ++ (0.02, 0) -- ++ (0, 0.02) -- ++ (-0.02, 0);

\node[right, darkgreen] at (2, 2.05) {\Large no forest dislocations};
\node[right, plotblue] at (2, 1.4) {\Large gradient flow of the $\Gamma$-limit};
\node[right, plotred] at (2, 1.1) {\Large limit of the gradient flows};

\end{tikzpicture}
\end{center}
\caption{\label{figure effective}We consider the normal velocity $V(f)$ of a single straight dislocation in the sharp interface limit without time rescaling under an applied exterior force $f$. In the figure above, both the force and the velocity of the interface are chosen to be positive in the direction of increasing slip. The three lines illustrate the kinetic relation derived from a viscous evolution without forest dislocations, the gradient flow of the Garroni-M\"uller $\Gamma$-limit (with $f_0=\Lambda\cdot\alpha(1)$), and the limit of gradient flows of the phase field model with forest dislocations, respectively. The regime of the dotted line has not been treated here, and is to be taken as a conjecture.}
\end{figure}

Details of the Theorem can be found in the main text, most importantly in Theorems \ref{theorem 2d slow}, \ref{theorem 2d torus} and Corollary \ref{corollary 2d general}. The exact time-scaling $c_\eps$ for a gradient flow $u_\eps$ of $\E_\eps$ for a single straight and aligned interface in the plane is not known, but the bounds
\[
\sqrt{\frac{\eps}{|\log\eps|^3}} \leq c_\eps \leq  \sqrt{\frac{\eps}{|\log\eps|}}
\]
hold. Stronger results are available in one dimension, see Theorems \ref{theorem slowness line}, \ref{theorem step line}, \ref{theorem step circle} and \ref{theorem external circle}. 

As seen in the heuristic calculations above or in other situations of `dynamic meta-stability' \cite{MR1075075}, the pinning constraint induces no motion on the macroscopic time-scale since its energy dissipation is highly localised at the obstacles. A similar phenomenon is observed in the simpler ODE model 
\[
\dot x = -2x\, \left(1- \sin\left(\frac {x^2}\eps\right)\right)
\]
which is the gradient flow of $F(x) = x^2 + \eps \cos (x^2/\eps)$. Short bursts of very fast motion can be observed here before getting trapped in a local energy minimum. Replacing the $\sin$-function by a suitable modification, we can instead observe very fast motion at steep drops alternating with very slow motion on almost flat segments. The overall motion becomes slow as $\eps\to 0$ due to the many flat segments of the potential. An energy dissipation argument for such a system can be found in \cite{MR2992854}.

Unfortunately, directly using energy dissipation techniques appears impossible in our model. Instead, we construct viscosity sub- and super-solutions of \eqref{eq gradient flow} which `trap' a solution. If we can establish a certain behaviour for both the sub- and super-solution, it follows that it also holds for the solution. Using suitable estimates, the slowness of sub- and super-solutions (or their behaviour according to kink/anti-kink attraction) can be established through a rigorous version of the heuristic calculation given above. For a one-dimensional problem without pinning, a similar approach has been used in \cite{gonzalez2010slow}.

The main difficulty in the proof therefore lies in the construction of suitable sub-solutions for the non-local evolution equation. We  first prove analogue statements in one dimension in Theorems \ref{theorem slowness line}, \ref{theorem step line}, \ref{theorem step circle} and \ref{theorem external circle} because this is technically easier and then modify the arguments to yield the result in two dimensions.

The construction proceeds in two steps. First we construct a stationary sub-solution at a pinning site by considering a periodic constrained minimisation cell problem and obtaining sharp decay estimates for the solution as $\eps\to0$. Then we carefully glue the stationary sub-solution to a modified optimal profile for the transition between the potential wells with precise estimates in order to not destroy the sub-solution property.

\section{One-dimensional Dynamics}\label{section 1d}
\label{section sub-solutions}

In this section, we  construct sub- and super-solutions to the evolution equation \eqref{eq gradient flow} in one space dimension for various initial conditions. This is significantly simpler than the two-dimensional case even at straight parallel interfaces, so we devote an entire section to demonstrate the techniques that will later be refined for the two-dimensional evolution. The rate we obtain is optimal in one dimension, while there is a difference of order $O(|\log\eps|)$ between the upper and the lower bound in the two-dimensional case.

\subsection{Periodic Obstacles}

We consider a rescaled version of the problem where a forest dislocation/obstacle has length scale $O(1)$. The same length scale occurs in the transition of an optimal profile between two potential wells.

\begin{lemma} \label{lemma obstacle sub-solution}
Denote by $S^1_l$ the circle of length $l\gg 0$ and take $R, M>0$ and an arbitrary point $x_0\in S^1_l$. Then, if $l\gg1$ is large enough, there exists a function $\bar u_l\in H^{1/2}(S^1_l)$ such that
\[
\bar u_l\equiv 0\text{ on }B_R(x_0)\qquad \text{and}\qquad \A \bar u_l - W'(\bar u_l) = \frac{M}{l} \text{ on }S_l^1\setminus B_R(x_0)
\]
in the weak sense. The function $\bar u_l$ has the following properties:

\begin{itemize}
\item $0\leq \bar u_l < 1$ and $\bar u_l>0$ outside $\overline{B_R(x_0)}$.

\item $\bar u_l\in C^{0,1/2}(S^1_l)$ and $\bar u_l \in C^\infty_{loc}(S_l^1\setminus \overline{B_R(x_0)})$.

\item If $x\in S^1_l$ and $\bar x$ denotes the reflection of $x$ through $x_0$, then $\bar u_l (\bar x) = \bar u_l(x)$.

\item If we identify $S^1_l = [-l/2, l/2)$ and $x_0=0$, then $\bar u_l$ is monotonically increasing on $[0,l/2)$. 

\item Let $\beta>0$. Then the set $\{\bar u_l<1-\beta\}$ is contained in $[-c_\beta, c_\beta]$ for some $c_\beta>0$ independently of $l$.

\item Let $-x_0$ denote the antipodal point of $x_0$ Then
\[
 u_l (-x_0) \leq 1 - \frac1{W''(0)} \left(\frac M{l} + \frac{R}{l^2}\right).
\]

\item There exist constants $c_1, c_2>0$ such that
\[
\ubar_l(x) \geq 1 - \frac{M}{W''(0)\,l} - \frac{c_2}{|x|^2}.
\]
\end{itemize}
All constants may depend on $R$. $M$ may depend on $l$ and the constants are uniform as long as $M_l\leq M_0$ is uniformly bounded.
\end{lemma}

\begin{proof}
{\bf Set Up.} For the time being, replace $W$ by a smooth double-well potential on $\R$ which agrees with the original on $[0,1]$ which is monotone and convex outside that interval such that $W'$ has linear and $W$ has quadratic growth at $\infty$ and such that $W(t)\geq t$ for $t\geq 2$. Once we see that all relevant functions take values only in $[0,1)$, we can pass back to the original multi-well potential. Consider the Hilbert space
\[
X_l:= \left\{u\in H^{1/2}(S^1_l)\:|\: u\equiv 0 \text{ on }B_1(x_0)\right\}.
\]
We will show that for every $l\gg 1$, the energy
\[
\E_{l}(u) = [u]_{1/2, S_l^1}^2 + \int_{S_l^1}W(u) \dx + \frac{M}l \int_{S_l^1} |u|\dx 
\]
has a minimiser $\ubar= \ubar_{l}$ in the open set
\[
U_l := \left\{u\in X\:\bigg|\:\frac1l \int_{S^1_l}u \dx > \frac12\right\}.
\]
Subsequently, we will show that $\ubar_l$ has the properties we claim in the Lemma.

{\bf Finite Energy.} First, we show that there exist functions $u_l\in U_l$ such that
\[
\E_l(u_l) \leq C
\]
for a universal constant $C$ (which depends on $R,M$). Take a smooth function $\eta:\R\to\R$ such that $\eta(r) = 0$ for $r\leq R$, $\eta(r) = 1$ for $r\geq 2R$ and monotone in between. Then set
\[
u_l(x) = \eta(d(x,x_0))
\]
where $d$ is the usual distance function on the circle. The function $u_l$ is smooth and has energy 
\begin{align*}
\E_{l}(u_l) &= \int_{S_l^1}\int_{S_l^1}K(x,y)\,|u_l(x)-u_l(y)|^2\dx\dy + \int_{S_l^1}W(u_l)\dx + \frac Ml\int_{S_l^1}u_l\dx\\
	&= \int_{B_{3R}(x_0)}\int_{B_{3R}(x_0)} K(x,y)\,|u_l(x)-u_l(y)|^2 \dx\dy\\
	&\qquad + \int_{S_l^1\setminus B_{3R}(x_0)}\int _{B_{2R}(x_0)} K(x,y)\,{|1-u_l(x)|^2}\dx \dy + 2\int_R^{2R} W(\eta) \dr + M\\
\end{align*}
Recall that the kernel $K$ of the $H^{1/2}$ semi-norm on $S^1_l$ is
\[
K(x,y) = \sum_{k\in \Z} \frac{1}{|x- (y+lk)|^{2}} = \frac1{l^2}\sum_{k\in \Z} \frac{1} { |\,(x-y)/l + k\,|^2} = \frac{1}{4\pi^2\,l^2\,\sin^2\left(\frac{x-y}{\pi l}\right)}
\]
when we identify $S_l^1 = (-l/2, l/2]$. To see this, take the circle as covered periodically by $\R$ and use that the half-Laplacian on the circle agrees with the half-Laplacian of the periodically lifted function on $\R$. Clearly, this gives the above kernel for the half-Laplacian and by extension for the $H^{1/2}$-semi-norm. Observe that $|x-y| \leq |x| + |y| \leq 3R + l/2$ for $x\in B_{3R}(0)$, $y\in S_l^1 = [-l/2, l/2]$, so
\begin{align*}
\left|\,K(x,y) - \frac{1}{|x-y|^{2}}\right| &= \sum_{k\in \Z\setminus \{0\}} \frac{1}{|x- (y+lk)|^{2}}\\
	&= \frac{1}{l^{2}} \sum_{k\in \Z\setminus \{0\}} \frac1{|(x-y)/l + k|^{2}}\\
	&\leq l^{-2} \max_{z\in B_{1/2 + 3R/l}(0)}\sum_{k\in \Z\setminus\{0\}}\frac{1}{|z-k|^{2}}\\
	&\leq C\,l^{-2}
\end{align*}
where $C$ is uniform for all $l\gg 1$. Using $|x-y|^2 \geq |y|^2/3$ for $x\in B_{2R}(x_0)$ and $y\in S^1_l\setminus B_{3R}(x_0)$, the remaining integral is then estimated by 
\begin{align*}
\int_{S_l^1\setminus B_{3R}(x_0)}\int _{B_{2R}(x_0)} K(x,y)\,{|1-u_l(x)|^2}\dx \dy &\leq \int_{S_l^1\setminus B_{3R}(x_0)}\int_{B_{2R}(x_0)}K(x,y)\dx\dy\\
	&\leq \int_{S_l^1\setminus B_{3R}(x_0)}\int_{B_{2R}(x_0)} \frac{1}{|x-y|^{2}} + C\,l^{-2} \dx \dy\\
	&\leq \int_{S_l^1\setminus B_{3R}(0)} \int_{B_{2R}(0)} \frac{3}{|y|^{2}} + C\,l^{-2} \dx\dy\\
	&= 2R\left(\int_{[-l/2, l/2]\setminus[-3R, 3R]}\frac{3}{|y|^2}\dy + C\,l^{-1}\right)\\
	&\leq C.
\end{align*}
Similarly, the local integral is uniformly bounded for large $l$ where the kernel approaches the kernel of the half-Laplacian on the real line.

{\bf Minimisers.} The direct method of the calculus of variations establishes that $\E_l$ has a minimiser $\bar u_l$ in the closure $\overline {U_l}$ of $U_l$. We will show that for large enough $l$, $\ubar_{l}$ must lie inside $U_l$. 

Assume that $1/2 < \frac1l \int_{S^1_l}\ubar_l\dx <4/7$. Then there are two possibilities:

\begin{itemize}
\item There exists a set $A_l\subset S^1_l$ such that $|A_l| \to\infty$ as $l\to\infty$ and $1/3\leq \ubar_l\leq 2/3$ on $A_l$ or

\item there exists no such set.
\end{itemize}

In the first case, we observe that
\[
\E_l(\ubar_l) \geq \int_{S^1_l}W(\ubar_l)\dx \geq |A_l|\,\min_{t\in [1/3, 2/3]} W(t)
\]
which goes to infinity when $l$ becomes large. Hence this is not possible for minimisers for large enough $l$. In the second case, we know that the sets
\[
B_l:= \{\ubar_l\geq 2/3\}, \qquad D_l:= \{\ubar_l\leq 1/3\} 
\]
satisfy $|B_l| + |D_l| \geq l-c$ for some constant $c>0$. Observe that $\bar u_l\geq 0$, since this cut-off decreases the energy and cannot violate the integral condition. So we deduce that
\begin{align}
4/7 \cdot l &>\int_{S^1_l}u\dx \geq 2/3\,|B_l|
\end{align}
Conversely, we know that 
\[
\int_{\{\ubar_l\geq 2\}} \ubar_l \dx \leq \int_{S^1_l}W(\ubar_l)\dx \leq C
\]
since $W(t)\geq t$ for $t\geq 2$. Therefore
\begin{align*}
\frac l2&< \int_{S^1_l}u\dx \leq 2\,|B_l| + C + \frac23\,|S^1_l\setminus (B_l\cup D_l)|  + \frac13\, |D_l|.
\end{align*}
We renormalise $b_l:=|B_l|/l$ and $d_l:= |D_l|/l$, so the inequalities read
\[
6/7 > b_l, \qquad 1/2 < 2 b_l + C/l + d_l/3.
\]
When we ignore the middle term in the second inequality and approximate $d_l= 1-b_l$, that simplifies to
\[
1/2 < 2\,b_l + (1-b_l)/3 = 5 b_l/3 + 1/3 \quad\LRa\quad 1/10 < b_l,
\]
so in total we have the weaker bounds
\[
1/11 < b_l < 6/7
\]
for large enough $l$, also taking into account the error term. Consequently, we obtain that $\delta < b_l, d_l <1-\delta$ for some $\delta>0$ and all large enough $l$. We can roll up the circle to an interval $I_l = [0, l]$ and use $K(x,y) \geq |x-y|^{-2}$. Then the re-arrangement result \cite[Proposition 6.1]{MR1657316} states that $B_l$ and $D_l$ are ideally distributed as two sub-intervals at opposite ends of $I_l$. We compute
\begin{align*}
\E_{l}(\ubar_l) &\geq [\ubar_l]_{1/2, S^1_l}^2\\
	&\geq \int_{B_l} \int_{D_l} \left(\frac13\right)^2\,K(x,y)\dx\dy\\
	&\geq \frac1{9} \int_{0}^{|B_l|}\int_{|B_l| + c}^{l} \frac{1}{|x-y|^2}\dx\dy\\
	&= \frac1{9} \int_0^{|B_l|} \frac1{|B_l|+c-y} -  \frac1{l-y} \dy\\
	&= \frac1{9} \left[\log(|B_l|+c) - \log(c) - \log(l) + \log(l-|B_l|)\right]\\
	&= \frac1{9} \log\left(\frac{(|B_l|+ c)(l-|B_l|)}{c\,l}\right)\\
	&\sim \log\left(\frac{|B_l|\,|D_l|}{|B_l| + |D_l|}\right).
\end{align*}
If both $|B_l|$ and $|D_l|$ go to infinity as $l\to\infty$ (as above), then $\E_l(\ubar_l)\to\infty$ as well, which leads to a contradiction. This key estimate is used in \cite[Lemma 4.5]{MR1657316} to establish $\Gamma$-convergence to the perimeter functional for the fractional Modica-Mortola energy with a double well potential and no pinning. Thus indeed $|D_l|\leq C$ and
\[
\frac1l \int_{S^1_l} \ubar_l\dx > 4/7,
\]
and $\ubar_{l}\in U_l$. We finally come to establishing the properties we claimed for $\ubar_l$.

{\bf Symmetry and Monotonicity.} Due to \cite[Theorem 3]{baernstein1994unified} $\ubar^* = \ubar$ agrees with its monotonically decreasing rearrangement around $-x_0$ since rearranging decreases the non-local term in the energy while leaving the local ones invariant and preserving the integral constraint. So, when we identify $S^1_l = (-l/2, l/2]$ and $x_0=0$ by the usual covering map, we see that $\ubar_l(x) = \ubar_l(-x)$ and $\ubar_l$ is monotonically increasing on $[0, l/2)$.
 
{\bf Growth.} When showing that $\ubar_l\in U_l$, we showed that one of the sequences of sets
\[
B_{ l}:= \{x\in S^1_l\:|\:\ubar_l\geq 2/3\},\qquad D_{l}:= \{x\in S^1_l\:|\: \ubar_l\leq 1/3\}
\]
must have uniformly bounded measures. Since $\frac1l\int_{S^1_l}\ubar_l\dx >1/2$, this can only be $B_l$. The argument can easily be generalised to see that the sequence of sets
\[
D_l^\beta := \{x\in S^1_l\:|\:\ubar_l\leq 1-\beta\}
\]
must have uniformly bounded measures depending on $\beta>0$. Since $\ubar_l$ is monotone growing away from $x_0 =0$, the sets $D_l^\beta$ are intervals and there exists a constant $c_\beta>0$ such that
\[
D_l^\beta \subset (-c_\beta, c_\beta) \subset (-l/2, l/2] = S^1_l
\]
for all sufficiently large $l$.

{\bf Boundedness.} We have seen that $\ubar_l\geq 0$ and that
\[
\frac1l\int_{S^1_l} \ubar_l \dx > \frac47, \qquad \frac1l\int_{\{\ubar_l\geq 2\}}\ubar_l\dx \leq \frac Cl.
\]
So clearly, $\min\{\ubar_l, 2\}$ satisfies the integral constraint for large enough $l$, vanishes on $B_1(x_0)$ and has strictly lower energy than $\ubar_l$ unless $\ubar_l\leq 2$ since we modified $W(t)$ to be monotonically increasing for $t\geq 1$. Thus $\ubar_l\leq 2$ for all sufficiently large $l$. 

{\bf Regularity.} Since $\ubar_l$ lies in the interior of $U_l$, we can calculate the change of energy in $\E_l$ under small variations of $\ubar_l$. $\E_l$ is smooth when we vary only where $\ubar_l>0$. Thus, we have the Euler-Lagrange equation
\[
\begin{cases}\hspace{3mm}\ubar_l = 0 &\text{ on }[-a_l,a_l]\\ \A \ubar_l = W'(\ubar_l) + M/l &\text{ in }S^1_l\setminus[-a_l,a_l]\end{cases}
\]
in a weak sense. The right hand side of the equation lies in $L^\infty(S^1_l)$ since $\ubar_l$ is bounded. Due to \cite{MR3161511}, $\ubar_l$ is a viscosity solution of the same equation and thus continuous. In fact,
\[
\ubar_l \in C^{0,1/2}(S^1_l)\cap C^\infty\left(S^1_l\setminus[-a_l,a_l]\right)
\]
due to \cite[Propositions 1.1 and 1.4, Theorem 1.2]{ros2014dirichlet}. The proofs in the literature are usually presented for bounded domains on Euclidean space, but also work in the periodic case.

{\bf Determining the Vanishing Set.} It remains to show that $a_l=1$.
Assume that $a_l>1$ and take $\phi \in C_c^\infty(1,a_l)$ with $\phi\geq 0$. Then we know that
\begin{align*}
0 &\leq \frac{\E_l(\ubar_l+t\phi) - \E_l(\ubar_l)}t\\
	&= \langle \ubar_l,\phi\rangle_{H^{1/2}} + \frac t2\,[\phi]_{H^{1/2}}^2 + \frac1t\int_{S_l^1}W(t\phi)\dx + \frac Ml\int_{S_l^1}\phi\dx\\
	&\to -\,\langle \A\ubar_l, \phi\rangle_{L^2} + \frac{M}l\int_{S_l^1}\phi\dx
\end{align*}
as $t\to 0$ since $W$ quadratically at zero and $W(\ubar_l+t\phi) = W(\ubar_l)$ where $\ubar_l\neq 0$. We can replace the $H^{1/2}$-inner product with an $L^2$-inner product since $\ubar_l \equiv 0$ is smooth on the support of $\phi$. It follows that $\A u_l \leq \frac{M}l$ on $(1,a_l)$, but this can easily be seen to be false for all large $l$ since
\[
\A\ubar_l(x) \geq \int_{c_\eta}^{l/2}\frac{\eta}{|y-x|^2}\dy \geq \frac{\eta}{c_\eta} - \frac{2\eta}{l} \longrightarrow\hspace{-4.4mm}\not\hspace{4.6mm} 0
\]
for all $\eta\in (0,1)$. Thus $a_l\equiv 1$ for $l\gg 1$.

{\bf Improved Boundedness.} By boundedness, symmetry and monotonicity, $\ubar_l$ is maximal at the antipodal point $-x_0$ of $x_0$. Due to smoothness, we can easily argue that $\A\ubar_l$ is defined pointwise around $-x_0$, i.e. 
\[
\A \ubar_l(-x_0) \leq \int_{-R}^R\frac{-\ubar_l(-x_0)}{|x + l/2|^2} \leq -\frac{R}{l^2}
\]
and thus
\[
-2R/l^2 \geq \A \ubar_l(-x_0) = W'(\ubar_l(-x_0)) + M/l
\]
holds pointwise. Since $W'>0$ outside $[0,1]$, this directly shows that
\[
\ubar\leq \ubar_l(-x_0) \leq 1- \frac1{W''(0)} \left(\frac M{l} + \frac{R}{l^2}\right) < 1
\]
for large enough $l$. From now on, we can use the original potential $W$.

{\bf Improved Growth.} Take $0<\beta<1/8$ such that $-W'(1-t) \geq \frac{W''(1)\,t}2$ for $t\in[0,\beta]$ and that $W'$ is monotonically increasing on $[1-\beta, 1]$. Assume that $S_l^1 = J \cup J^c$  and $w_l\in C^0(S_l^1) \cap H^{1/2}(S_l^1)$ such that $\ubar_l\geq w_l$ on $J$ and 
\[
\A w_l - W'(w_l)\geq \frac{M}l\quad \text{as well as}\quad \ubar_l\geq 1-\beta \text{ on }J^c.
\]
Then by a simple application of the maximum principle we have $\ubar_l \geq w_l$ on $S^1_l$. Assume the contrary. Then $w_l-\ubar_l$ has a positive maximum somewhere in $J^c$. We find that
\[
0\geq \A(w_l-\ubar_l) \geq W'(w_l) + M/l - W'(\ubar_l) -M/l = W'(w_l) - W'(\ubar_l) >0
\]
at this point since $1\geq w_l>\ubar_l\geq 1-\beta$ are both in the area where $W'$ is monotonically increasing and obtain a contradiction. We will now construct comparison functions $w_l$. 

By monotonicity and growth beyond $1-\beta$ on uniformly finite intervals, there exists an interval $J_\beta$ around $x_0$ such that $\ubar_l\geq 1-\beta$ outside of $J_\beta$ independently of $l$. Take $l\gg 1$ and define
\[
w_l:S_l^1=(-l/2,l/2]\to \R, \qquad w_l(x) = \left(1- \frac{\gamma M}l - \frac{c_2}{|x|^2}\right)_+.
\]
We easily calculate 
\[
w_l(x) = 0 \quad \LRa\quad |x| \leq \sqrt{\frac{c_2}{1- \gamma M/l}}, \qquad w_l(x)\geq 1-\beta \quad\LRa\quad |x|\geq \sqrt{\frac{c_2}{\beta - \gamma M/l}} \geq 2\,\sqrt{2\,c_2},
\]
so in particular $\w_l$ vanishes on an interval 
\[
[-\sqrt{c_2} , \sqrt{c_2} ] \subset J \subset [-\sqrt{2c_2}, \sqrt{2c_2}].
\]
In particular we can choose $c_2$ so large that $\ubar_l(x) \leq 1-\beta$ implies $w_l(x) = 0$. So we observe that
\begin{align*}
\A w_l (x) - W'(w_l(x)) &= \A w_l(x) - W'\left(1 - \frac{\gamma M}l - \frac{c_2}{|x|^2}\right)\\
	& \geq \A w_l(x) + \frac{W''(1)}{2} \left(\frac{\gamma M}{l} + \frac{c_2}{|x|^2}\right)\\
	&\geq \frac{W''(1)\,\gamma}{2}\,\frac Ml + \left(\A w_l(x) + \frac{W''(1)\,c_2}{2\,|x|^2}\right)
\end{align*}
whenever $w_l\geq 1-\beta$. Therefore we can choose $\gamma \geq 2/W''(1)$ and only need to show that the second term is non-negative for $|x| \geq 4\sqrt{c_2}$. Without loss of generality, take $x>0$. Now we observe that $w_l(x) =0 \Ra |x|\leq\sqrt{2\,c_2}$ and recall that $\frac{c }{|x|^2} \leq K_l(x) \leq \frac{C}{|x|^2}$ where $K_l$ is the kernel of the half-Laplacian on $S^1_l = (-l/2, l/2]$ and the constants $c, C$ are uniform in $l\gg 1$.

First, let us assume that $0 < x < l/2 -1$. We disregard the integrals with the right sign except for the one over $[x, x+1)$ and count the ones pulling down twice when the are over a subset of $[0, l/2)$ instead of considering the ones from $(-l/2, 0]$. Since the kernel $K_l$ is monotone, this is admissible. Pick $2/3 < \alpha <1$ and compute
\begin{align*}
\A w_l(x) &=  P.V.\int_{-l/2}^{l/2} [w_l(y) - w_l(x)]\,K_l(x-y)\dy\\
	&\geq - \int_{\sqrt{2c_2}}^{\sqrt{2c_2}}K_l(x-y)\dy + 2\,\int_{\sqrt{2c_2}}^{x^\alpha } K_l(x-y)\, [w_l(y)-w_l(x)]\dy \\
	&\qquad  + 2\,\int_{x^\alpha}^{x-1} K_l(x-y)\int_x^y w_l'(t)\dt \dy + \int_{x-1}^{x+1} K_l(x-y) \int_x^y(y-t)\,w_l''(t)\dt\dy \\
	&\geq -C\left(\frac{\sqrt{2c_2}}{|x-\sqrt{2c_2}|^2} + \frac{1}{|x-x^\alpha|^2}\int_{c_2}^{x^\alpha}\frac{c_2}{y^2}\dy + \frac{c_2}{|x^\alpha|^3}\int_{x^\alpha}^{x-1} \frac{1}{|y-x|}\dy + \frac{c_2}{|x|^4}  \right)
\end{align*}
In the first term, we used that the jump is at most $1$, in the second we pulled out the integral kernel and used only the negative part of the difference. If $x^\alpha\leq \sqrt{2c_2}$, the second term simply disappears. In the third term, we kept the integral kernel and used the largest possible value of the derivative, and the fourth term is estimated solely by the second derivative of $|x|^{-2}$. Constants were absorbed into $C$. Thus
\[
\A w_l(x) \geq  -C\left(\frac{\sqrt{2c_2}}{|x-\sqrt{c_2}|^2} + \frac{1}{|x-x^\alpha|^2} + \frac{c_2 \log(x)}{|x^\alpha|^3} + \frac{c_2}{|x|^4}  \right)
\]
When we choose $c_2$ large enough, we can see that 
\[
\A w_l(x) \geq - \frac{W''(1)\,c_2}{2\,|x|^2}
\]
for all $|x|\geq \sqrt{2\,c_2}$. Finally, we observe that the argument can also applied for $|x|>l/2-1$ when we replace $1/x^2$ by a function $f_l\in C^2(S_l^1)$ satisfying
\[
\frac c{|x|^2} \leq f_l(x) \leq \frac{C}{|x|^2}, \qquad |f_l'(x)| \leq \frac{C}{|x|^3}, \qquad |f_l''(x)|\leq \frac {C}{|x|^3}.
\]
This establishes the growth estimate for $2/W''(1)$. With this estimate, we can go back and improve the growth to $1/W''(1)$.
\end{proof}

\begin{remark}\label{remark decay 2D}
The same method can be applied on tori $T_l^d = (S_l^1)^d$ in any dimension $d\geq 1$, but only yields solutions to
\[
\A u - W'(u) = \frac{M}{l^d}
\]
with faster decaying constant on the right hand side. That decay rate is not sufficient for later applications, since an interface along a straight line is essentially one-dimensional and exerts a force of order $1/l$. 

The case $d=1$ is special in the proof above since $S^d_l = T^d_l$. On tori in higher dimensions, the re-arrangement is significantly more involved, since for example the monotonically increasing and monotonically decreasing rearrangements do not agree. 
\end{remark}

\begin{remark}
Assume that we are instead constructing an obstacle with boundary conditions at $a\in \Z$, say $a\geq 2$. Then we modify $W$ outside $[0,a]$ instead and obtain the same results as before under the integral side condition $\frac1l \int_{S^1_l} u\dx > a-1/2$. The proof is only slightly more involved. 
\end{remark}

For technical purposes, it may be helpful to continue the solutions onto a larger set.

\begin{lemma}\label{extension lemma}
Let $m\in\N, m\geq 2$ and $u\in C^2(S^1_l)$. Identify $S^1_r = (-r/2, r/2]$ and define
\[
u_{l, ml}: S^1_{ml}\to \R, \quad u_{l,ml} (x) = \begin{cases}u(x) &x\in S^1_l\\ u(l/2) &x\in S^1_{ml}\setminus S^1_l\end{cases}.
\]
If $u$ is maximal at $l/2$, then 
\[
\A^{S^1_{ml}} u_{l,ml}(x)\geq \A^{S^1_l}u(x)
\]
for $x\in S^1_l$ and
\[
\A^{S^1_{ml}}u_{l, ml}(x) \geq \A^{S^1_l}u(l/2)
\]
for $x\in S^1_{ml}\setminus S^1_l$. The same holds true for $m=\infty$ (i.e.\ $S^1_{ml} = \R$).
\end{lemma}

\begin{proof}
We can consider $u$ as a function on $S^1_{ml}$ with period $l$ and the fractional Laplacians of these functions agree. The singular integral clearly becomes larger when we replace all the periods but one with a constant function at the maximum. On the constant segment of the new function, the part of the singular integral pulling downward is no larger than before on the circle at the highest point, since there is a smaller set contributing to the integral pulling down, and it is farther away.
\end{proof}

We formulated the Lemma in the setting of smooth functions to compute the singular integral directly, but by density it also holds in the distributional sense for $u\in H^{1/2}(S^1_l)$, thus in particular
\[
\A^{S^1_{ml}}\ubar_{l,ml} - W'(\ubar_{l,ml}) \geq \frac1l
\]
on $\{\ubar_{l,ml}>0\} = S_{ml}^1\setminus B_R(x_0)$. 

The same methods as in Lemma \ref{lemma obstacle sub-solution} can be used to establish the following result for a global minimiser without the term forcing downwards.

\begin{lemma}\label{lemma subsolution line}
Let $u\in 1+ H^{1/2}(\R)$ be a minimiser of 
\[
\E(u) = [u]_{1/2}^2 + \int_\R W(u)\dx
\]
under the constraint $u\equiv 0$ on $[-R,R]$. Then $u$ satisfies $1 - \frac{c}{|x|^2} \leq u(x)$ for all $x\in \R$ and some $c\geq 1$.
\end{lemma}

\subsection{The Interface} \label{section interface}

Recall the following results for transitions between potential wells.

\begin{lemma}\cite{cabre2005layer}\label{lemma interface}
There exists a function $\phi \in C^{2,\alpha}(\R)$ such that $\phi$ is monotonically increasing,
\[
\lim_{x\to\infty}\phi(x) = 1, \qquad \lim_{x\to-\infty}\phi(x) = 0, \qquad \A \phi - W'(\phi) = 0.
\]
The function satisfies
\[
\frac{c}{1+x^2} < \phi'(x) \leq \frac{C}{1+x^2}
\]
for some $C\geq c>0$.
\end{lemma}

The estimate on the derivative further implies that
\[
\phi(x) = 1-\int_x^\infty \phi'(t)\dt \leq 1- \int_x^\infty \frac{C}{1+t^2}\dt \leq 1- \frac{C'}x
\]
for any $C'>C$ and sufficiently large $x$, as well as $\phi(x) \geq 1- \frac{c}{x}$ for all sufficiently large $x$. This has been sharpened to the estimate
\[
\left|1- \frac{1}{W''(0)\,x} - \phi(x)\right| = O(x^{-2})
\]
in \cite[Theorem 3.1]{gonzalez2010slow}. The same decay holds for large negative $x$:
\[
\left|\phi(x) - \frac{1}{W''(0)\,|x|}\right| = O(|x|^{-2}).
\]
Note that the constant in \cite{gonzalez2010slow} is slightly different since the operator used there is the half-Laplacian in its usual normalisation, while we neglected a dimensional constant for easier notation. Under additional conditions, we can also control the second derivative of $\phi$. Note that for the popular choice 
\[
W(u) = \frac1\pi \,\left[\cos(\pi u) + 1\right]
\]
(with wells on $\Z+ 1/2$) we have the transition function
\[
\phi(x) = \frac2\pi\,\arctan(x), \qquad\phi'(x) = \frac{2/\pi}{1+x^2}, \qquad \phi''(x) = \frac{4x}{\pi\,(1+x^2)^2}
\]
so that also 
\begin{equation}\label{eq: decay interface}
|\phi''|\leq C/(1+x^2)^{3/2}.
\end{equation}
In the following, we assume that $W$ is chosen such that the optimal transition function $\phi = \phi_W$ satisfies \eqref{eq: decay interface}.

\begin{lemma}\label{lemma modified interface}
Let $L\gg 1$, $\beta$ as in Lemma \ref{lemma obstacle sub-solution}. Then there exists function $\widetilde\phi = \widetilde\phi_L\in C^{2,\alpha}(\R)$ such that
\begin{enumerate}
\item $\widetilde\phi$ is monotone increasing,
\item $\widetilde\phi\equiv 0$ on $(-\infty, -L/W''(0) + \tilde C)$,
\item $\widetilde\phi\equiv 1 - \frac1L$ on $[L/W''(0) + \tilde C,\infty)$,
\item whenever $0<\widetilde\phi(x) < \beta$ or $1-\beta < \widetilde\phi(x) \leq 1$ we have
\[
\left(\A\widetilde\phi - W'(\widetilde\phi)\right)(x) \geq \frac{\bar c}{L^2}
\]
\item whenever $\beta \leq \widetilde\phi(x) \leq 1-\beta$, we have
\[
\left|\A\widetilde\phi - W'(\widetilde\phi)\right|(x) \leq \frac{C}{L^2}.
\]
\end{enumerate}
The constants $\bar c, C, \tilde C$ depend on $W$, but not on $L$.
\end{lemma}

\begin{proof}
Take $f_L:\R\to\R$ to be a smooth function such that
\[
f_L(z) = \begin{cases}0 & z \leq \frac1L\\
			 z - \frac{C}{L^2} &\frac2L \leq z \leq 1 - \frac2L\\
			 1 - \frac1L & z\geq 1-\frac1L
	 	\end{cases}
\]
and $0 \leq f_L' \leq 3$, $|f_L''|\leq \frac{10}{L}$ and define
\[
\widetilde\phi = f_L\circ \phi.
\]
We see that 
\[
\widetilde\phi' = (f_L'\circ\phi)\,\phi'\geq 0,
\]
so $\widetilde\phi$ is monotone increasing. Furthermore, we obtain that
\[
\phi(x) \geq 1 - \frac{1}{W''(0)\,x} - \frac{C}{x^2}  \geq 1- \frac1L \qquad\forall\ x\geq L/W''(0) + C
\]
and thus $\widetilde\phi(x) \equiv 1- \frac1L$ for all $x\geq L/W''(0) + C$. Analogously, $\widetilde\phi(x) \equiv 0$ for all $x\leq - (L/W''(0)+C)$. Now compute
\begin{align*}
(\widetilde\phi - \phi)' &= [(f_L'\circ\phi) - 1]\,\phi'\\
(\widetilde\phi - \phi)'' &= [(f_L'\circ\phi) - 1]\,\phi'' + (f_L''\circ\phi)\,(\phi')^2. 
\end{align*}
Thus it is easy to see that
\[
\left|\widetilde\phi-\phi\right|\leq \frac{2}L, \qquad \left|\widetilde\phi'-\phi'\right| \leq \frac{C}{L^2}, \qquad \left|\widetilde\phi'' - \phi''\right|\leq \frac{C}{L^3}.
\]
When we abbreviate $w= \widetilde\phi-\phi$, we can therefore use a representation of $\A$ like \eqref{eq fracoperator form} to compute
\begin{align*}
\A(\widetilde\phi-\phi)(x) &= \int_{x-L}^{x+L}\frac{\int_x^y(y-t)w''(t)\dt}{|y-x|^2}\dy +\int_{\R\setminus[x-L,x+L]}\frac{w(y)-w(x)}{|y-x|^2}\dy\\
	&\leq \int_{x-L}^{x+L} ||w''||_{L^\infty}\dy + 2\int_L^\infty \frac{||w||_\infty}{y^2}\dy\\
	&\leq \frac{C}{L^2}.
\end{align*}
Finally, take $x$ such that $\phi(x)\geq 1-\beta$ or $0<\phi_L(x) \leq \beta$. Then we can use that $\tilde \phi \leq \phi - \frac{C}{L^2}$ or $\wt\phi =0$ to estimate
\begin{align*}
-W'(\widetilde\phi(x)) &= -W'(\widetilde\phi(x) - \phi(x) + \phi(x))\\
	&\geq - W'(\phi(x)) - \frac{2\,W''(\phi(x))}3\, (\widetilde\phi(x) - \phi(x))\\
	&\geq - W'(\phi(x)) - \frac{W''(1)}2\,\frac{C}{L^2}.
\end{align*}
Thus in total
\[
\A\widetilde\phi - W'(\widetilde\phi) \geq \frac{C}{L^2}
\]
if $\widetilde\phi(x)\notin[\beta, 1-\beta] \cup\{0\}$ and 
\[
\left|\A\widetilde\phi - W'(\widetilde\phi)\right|\leq \frac{C}{L^2}.
\]
\end{proof}

\subsection{Dynamics on the Real Line I}

We can apply the sub-solutions constructed above to the one-dimensional problem by glueing a sub-solution for periodic obstacles to that for an interface. First we describe our results on the real line for a single step and then for a kink/anti-kink pair. The first case cannot be achieved with finite energy, but it helps us identify the time scale on which the pinning constraint induces motion. Here, we need to use solutions to \eqref{eq gradient flow} in the viscosity sense since a single transition layer does not have finite energy.

\begin{theorem}[A single step]\label{theorem slowness line}
Let $x_{i,\eps} = i\,d_\eps$ for $d_\eps\gg \eps$. Then there exist $\ul u_\eps \leq \ol u_\eps$ which are a viscosity sub- and super-solution of
\begin{equation}\label{eq single step evolution}
\begin{pde}
c_\eps \eps \,u_t &= \frac{1}{|\log\eps|}\left(\A u - \frac1\eps\,W'(u)\right) &\text{in }\R\setminus \bigcup_{i\in \Z}\overline{B_\eps(id_\eps)}\\
	u& =0 &\text{on }\bigcup_{i\in \Z}\overline{B_\eps(id_\eps)}
\end{pde}
\end{equation}
respectively with the following property: When we choose $c_\eps = {}\frac\eps {d_\eps^2\,|\log\eps|}$, there are constants $c, C>0$ such that
\[
\lim_{\eps\to 0} \ul u_\eps(t, \cdot) = \chi_{[ct,\infty)}, \qquad \lim_{\eps\to 0}\ol u_\eps(t,\cdot) = \chi_{[Ct,\infty)}
\]
in $L^2_{loc}(\R)$ for all $t>0$.
\end{theorem}

In one dimension, the natural obstacle scale is $d_\eps \sim 1/|\log\eps|$, so the gradient flow equation is slow on a scale of $O(\eps\,|\log\eps|)$. The interface moves with speed $O(1)$ if $d_\eps \sim \sqrt{\eps/|\log\eps|}$, which is the natural distance  between obstacles in two space dimensions; it is slow for larger distances and fast for smaller ones. It moves on the same time-scale as an interface would due to the kink/anti-kink attraction for $d_\eps\sim \sqrt\eps$. The proof also goes through for $d_\eps = N\eps$ for large enough $N$. 

The assumption that the obstacles are distributed on a lattice can of course be weakened significantly, and we believe that solutions $u_\eps$ should in fact converge to a characteristic function with a linearly propagating front $\chi_{[vt,\infty)}$ for periodic obstacles. We do not pursue these questions further.

\begin{proof}[Proof of Theorem \ref{theorem slowness line}]
{\bf Construction of a sub-solution.} For convenience, we build the sub-solution in the blow-up scale. Choose $l = l_\eps = d_\eps/(\eps N)$ for some sufficiently large $N\in\N$ (to be specified later), $R=1$ and $M=1$. We can extend the sub-solution for an obstacle $\ubar_l$ on the circle $S_l^1$ from Lemma \ref{lemma obstacle sub-solution} to the circle of length $Nl$ as in Lemma \ref{extension lemma}. 

Define $L$ by $1-\frac1L = \ubar_l(-x_0)$ and take $\widetilde\phi$ associated to $L$. By growth estimates for the optimal profile and the obstacle cell solution, $L = l + O(1)$. Furthermore, by this we know that $\widetilde\phi$ is constant on intervals $(-\infty, l/W''(0)+\tilde C]$ and $[l/W''(0)+\tilde C,\infty)$ for potentially larger constants $\tilde C$.  While $l$ and $L$ depend on $N$, the constants are uniform (at least for $l>l+0$ bounded from below), and we can take $N$ such that $N> 4C$. We will place an additional condition on $N$ later. 

Our sub-solutions are given by a modified transition $\wt\phi$ which moves with speed $\alpha$ in the space between two obstacles. Once we get too close to an obstacle, we jump over it instantaneously. In formulas
\[
\tilde u(t, x) = \begin{cases} \widetilde\phi\left( x - \frac{N}4l -\alpha t  \right) & x \leq \frac{3N}4l \\
	\ubar_{l,Nl}(x)&x\geq \frac{3N}{4l}
	\end{cases}
\] 
for $t\in [0, \frac{Nl}{4\alpha}]$ and
\[
\tilde u(t, x) = \tilde u\left(t - m\frac{Nl}{4\alpha}, \: x - mNl \right)\qquad\text{for } t\in \left( m\frac{Nl}{4\alpha}, (m+1)\frac{Nl}{4\alpha}\right], \quad m\in\N.
\]
 By construction, $u$ is continuous in space for all times $t$, jointly upper semi-continuous and non-increasing in time for a fixed point $x\in\R$. Since for fixed $x$ we only jump down as time increases, $u$ is clearly a sub-solution at the points of discontinuity in time. It remains to find $\alpha$ such that $u$ is a sub-solution also where it evolves smoothly.

At smooth points of $\tilde u$ away from the pinning set $\bigcup_{i\in Z} B_1(i\,Nl)$, it is sufficient to verify the inequality $u_t \leq \A u- W'(u)$ pointwise to obtain that $\tilde u$ is a viscosity sub-solution.

At points where $\widetilde\phi=0$, $\tilde u$ is clearly a sub-solution as it is constant in time, $W'=0$ and the fractional Laplacian pulls upwards at the minimum value. At points where $\tilde u= u_{l,Nl}$, $\tilde u$ is a sub-solution since the interface exerts a downward force proportional to its inverse distance to an obstacle. Since the obstacles are constructed to compensate a pressure of $1/l$ and the interface does not come closer than $\frac N8 l$, we can choose $N$ sufficiently large to make sure that the obstacles compensate this pressure. 

Finally, take $(t, x)$ such that $\tilde u(t, x) = \widetilde\phi(t, x)$ and compute
\begin{align*}
\A \tilde u(t, x) - W'(\tilde u(t, x))  &= \A (\tilde u - \widetilde\phi)(t, x) + \A(\widetilde\phi(t,x)) - W'(\widetilde\phi (t, x))\\
	&= \A (\tilde u - \widetilde\phi)(t, x) + O(l^{-2})\\
	&= \int_{x + \frac{N}8l}^\infty \frac{(\tilde u - \widetilde\phi)(y)}{|y-x|^2} \dy + O(l^{-2})\\
	&\geq - \sum_{i=1}^\infty \int_{i\,Nl -1}^{i\,Nl+1} \frac{1}{|y- x|^2}\dy + \int_{(i-1)\,Nl}^{i\,Nl-1}\frac{c\,|y- i\,Nl|^{-2}}{|y-x|^2}\dy\\
	&\hspace{3cm} + \int_{i\,Nl+1}^{(i+1)Nl} \frac{c\,|y- i\,Nl|^{-2}}{|y-x|^2}\dy + O(l^{-2})\\
	&\geq -\sum_{i=1}^\infty \frac{2}{|i\,Nl - Nl/4|^2} + \frac2{|(i-1 -1/4)\,Nl|^2}\int_1^l\frac{1}{y^2}\dy + O(l^{-2})\\
	&= O(N^{-2}l^{-2}) + O(l^{-2})\\
	&= O(l^{-2}).
\end{align*}
Now we can finally use that the second $O(l^{-2})$ term is positive where $\widetilde\phi\in(0,\beta]$ or $[1-\beta,1]$, and it compensates the first term for large enough $N$ (which can be chosen independently of $l$). At the interface we have
\[
\bar c:= \max_{z\in [\beta,1-\beta]}\widetilde\phi'(z)>0,
\]
so we can choose $\alpha = O(l^{-2})$ such that $\tilde u$ is a sub-solution. 

{\bf Rescaling.} Let us pass back to the original length scale:
\[
\ul u_\eps(t,x) = \tilde u\left( \frac{t}{c_\eps\eps^2\,|\log\eps|}, \frac{x}\eps\right).
\]
By construction, $\ul u_\eps$ is a sub-solution of \eqref{eq single step evolution} since (at smooth points)
\begin{align*}
c_\eps\eps\,\partial_t\ul u_\eps &= \frac{c_\eps\,\eps}{c_\eps\,\eps^2\,|\log\eps|} (\partial_t\tilde u)\\
	&\leq \frac{1}{\eps\,|\log\eps|}\left((\A\tilde u) - W'(\tilde u)\right)\\
	&= \frac1{|\log\eps|}\left(\A\ul u_\eps - \frac1\eps\,W'(\ul u_\eps)\right).
\end{align*}
We know that the interface in the blow up scale moves by exactly $Nl_\eps$ in time $Nl_\eps/(4\alpha) \sim N\,l_\eps^3$, so the rescaled interface moves by $d_\eps = N\eps l_\eps$ at the time $t_\eps$ such that
\[
\frac{t_\eps}{c_\eps\,\eps^2\,|\log\eps|} \approx N\,\left(\frac{d_\eps}\eps\right)^3\qquad \LRa\qquad \frac{t_\eps}{c_\eps} = \frac{N\,d_\eps^3\,|\log\eps|}\eps.
\]
To obtain a speed of $O(1)$, we need $t_\eps\sim d_\eps$, so we choose the acceleration factor
\[
c_\eps = \frac{\eps}{d_\eps^2\,|\log\eps|}.
\]

{\bf Limiting behaviour.} In the limit $\eps\to 0$, the jumps over shorter and shorter spatial intervals disappear and $\ul u_\eps(t,\cdot)$ converges locally in $L^1$ to the characteristic function of an interval $I(t)$ moving with uniform speed. If $c_\eps$ is chosen too small, then $I(t) =\emptyset$ for all positive times, whereas too large $c_\eps$ implies $I(t) \equiv [0,\infty)$. In the scaling regime identified above, we have $I(t) = [ct,\infty)$ for some $c>0$. 

{\bf Construction of super-solutions.} Here we work directly on the macroscopic scale. Let us make the ansatz
\[
\ol u_\eps(t, x) = \min\left\{\phi\left(\frac{x-\alpha t}\eps \right), \ubar_{l_\eps}\left(\frac x\eps\right)\right\}.
\]
In the stationary case $\alpha=0$ as the minimum of two solutions this is clearly a super-solution. Still for small positive $\alpha$, it suffices to consider $(t, x)$ such that $\ol u_\eps(t, x) = \phi(\dots)$ since at other points (including the non-smooth points where $\phi$ and $\ubar_{l_\eps}$ meet) the super-solution property is still easily established. The function is continuous by construction and satisfies the pinning constraint. Finally, compute 
\begin{align*}
c_\eps\eps \,\partial_t\ol u_\eps(t,x) 
	&= -\alpha\,\frac{\eps}{d_\eps^2\,|\log\eps|}\,\phi'\left(\frac{x-\alpha t}\eps\right)\\
	&\geq \frac{-c\alpha\eps}{d_\eps^2\,|\log\eps|}\,\frac{1}{\left(\frac{x-\alpha t}{\eps}\right)^2+1}\\
\frac{1}{|\log\eps|}\left(\A\ol u_\eps - \frac1\eps\,W'(\ol u_\eps)\right)
	&= \frac1{\eps\,|\log\eps|}\left((\A\phi) - W'(\phi)\right) + \frac1{|\log\eps|}\left(\A\ol u_\eps - \frac1\eps (\A \phi)\right)\\
	&= \frac1{|\log\eps|}\A \left(\ol u_\eps - \phi\left(\frac{\cdot-\alpha t}\eps\right)\right)\\
	&\leq \frac{1}{|\log\eps|}\sum_{i\in\Z} \int_{[id_\eps-\eps, id_\eps+\eps]}\frac{-\phi\left(\frac{y-\alpha t}\eps\right)}{|y-x|^2}\dy\\
	&\leq - \frac{2\eps\,\phi\left(\frac{x-\alpha t-d_\eps}\eps\right)}{|\log\eps|\,d_\eps^2}
\end{align*} 
by just considering the index $i\in \Z$ such that $x-d_\eps \leq id_\eps \leq x$. Since $\phi(z) \geq c\,\min\{1, 1/|z|\}$ vanishes more slowly than $\phi'$, this shows that
\[
c_\eps\eps \,\partial_t\ol u_\eps(t,x) \geq \frac{1}{|\log\eps|}\left(\A\ol u_\eps - \frac1\eps\,W'(\ol u_\eps)\right)
\]
for suitably small $\alpha$ which is independent of $\eps>0$.
\end{proof}

\begin{remark}
It is possible to prove a comparison principle for the evolution equation \eqref{eq gradient flow} in the viscosity sense. This has been done for equations on the whole space and operators of the type $(-\Delta)^s$ for $s>1/2$ in \cite{MR2121115}, but the methods go through for $s\leq 1/2$ and equations on domains with only minor modifications. Thus, the existence of a viscosity solution $u_\eps$ with $\ul u_\eps \leq u_\eps \leq \ol u_\eps$ follows directly by Perron's method. For a viscosity solution with given initial data, additional barriers have to be constructed. It is well known that $u$ solves 
\[
\begin{pde}
c_\eps\eps\,u_t &= \frac1{|\log\eps|}\left(\A u - \frac{W'(u)}\eps \right) &t>0, x\in \Omega_\eps \\
		u &\equiv 0& t\geq 0, x\in \Omega_\eps^c\\
		u &= u^0 & t=0
\end{pde}
\]
 in the viscosity sense if and only if $v\coloneqq e^{-\lambda t}u$ solves an equation of the form with the non-linearity
 \[
 f_\lambda(t,v) = -\frac{e^{-\lambda t} W\left(e^{\lambda t}v\right)}\eps - \lambda v.
 \]
in place of $W'(u)/\eps$. When we choose $\lambda$ large enough (depending on $W$ and $\eps$), the function $f_\lambda$ is monotone in $u$ uniformly in $t$. For an initial condition $u^0$ such that 
\[
\A u - \frac{W'(u)}\eps \in L^\infty\left(\Omega_\eps\right),
\]
we can then construct sub- and supersolutions by
\[
\ol v(t,x) = u_0(x) + Ct\cdot \chi_{\Omega_\eps}, \qquad \ol v(t,x) = u_0(x) - Ct\cdot \chi_{\Omega_\eps}
\]
for some large constant $C>0$. This includes all initial conditions in $C^2_b(\R)$ and all initial conditions that we are interested in. Since $v$ attains the initial condition, also $u$ does. The domain $\Omega_\eps$ can be chosen to be periodic or the perforated real line $\R\setminus B_\eps(d_\eps\cdot\Z)$.
\end{remark}

\subsection{The Corrector} \label{section corrector}

We have seen that the pinning constraint induces motion on a time-scale which is strictly slower than the $\log\eps$-timescale on which the {next-order term in a $\Gamma$-expansion \eqref{eq first order gamma}} acts as a kink/anti-kink attraction. We want to show that the pinning does not affect the attraction and annihilation of a single kink/anti-kink pair. For this purpose, we need a more refined construction to obtain the exact speed of an interface rather than just the order in $\eps$. Therefore, we need to know the behaviour of a moving interface to the next order. 

Set
\[
\eta:= \frac1{W''(0)}\int_{-\infty}^\infty \left(\phi'\right)^2\dx.
\]

\begin{lemma}
There exists a function $\psi \in H^{1/2}(\R) \cap C^{1,\alpha}(\R) \cap L^\infty(\R)$ for some $\alpha>0$ which solves
\[
\A \psi - W''(\phi)\psi = \phi' + \eta\,\left(W''(\phi) - W''(0)\right).
\]
The solution $\psi$ satisfies the estimate
\[
|\psi'(x)| \leq \frac{C}{1+x^2}
\]
and if $W\in C^{3,1}(\R)$ then also $\psi \in C^{2,\alpha}(\R)$ for some $\alpha>0$ and
\[
|\psi''(x)| \leq \frac{C}{1+x^2}.
\]
\end{lemma} 

The Lemma is proved in \cite[Theorem 3.2]{gonzalez2010slow} without the decay estimate, see also \cite[Lemma 2.2]{Patrizi:2016aa}. A simple proof goes as the one for the decay estimate on $\phi'$.

\begin{proof}[Idea of Proof:]
We only sketch the proof of the decay estimates. Consider the case $W_\alpha(z) = \frac{1- \cos(\pi z)}{\pi^2 \alpha}$ for $\alpha>0$ which has the explicit solution $\phi_\alpha(t) = \frac2\pi \,\arctan(\frac t\alpha)$ (with potential wells at $\pm 1$ instead of $0$ and $1$). We note that the derivative of any optimal transition $\phi$ satisfies
\[
\A\phi' - W''(\phi)\phi' = 0
\]
so in particular 
\[
\A (\phi_\alpha') + \frac1\alpha \,\phi_\alpha' \geq 0
\]
for all large (positive and negative) $x$. Given another potential $W$, we split $W''(\phi) = f_+(x) + f_-(x)$ with $\inf f_+>0$ and $f_-$ compactly supported. This splitting allows us to use a comparison with the solutions $\phi_\alpha$ for a suitable $\alpha$, taking the compactly supported `bad' term to the other side. We calculate formally 
\begin{align*}
\A \psi - W''(\phi)\psi &= \phi' + \eta\,\left(W''(\phi) - W''(0)\right)\\
\A \psi' - W''(\phi)\psi' &= W^{(3)}(\phi)\phi'\psi + \phi'' + \eta\,W^{(3)}(\phi)\phi'\\
\A \psi'' - W''(\phi) \psi'' 
	&= W^{(3)}(\phi)\left\{2\phi'\psi' + \phi''\psi + \eta \phi''\right\} + W^{(4)}(\phi)\left\{(\phi')^2\psi + \eta(\phi')^2\right\}.
\end{align*}
If $W\in C^{3,1}(\R)$, then the last equation makes sense with a right hand side in $L^\infty(\R)\cap L^1(\R)$ and the regularity of $\psi$ can be improved to $C^{2,\alpha}(\R)$. The decay now follows as in the proof of \cite[Theorem 1.6]{cabre2005layer}.
\end{proof}

Note that due to the decay estimate on the derivative, $\phi \pm \eps \psi$ is still monotone increasing for all small enough $\eps>0$. As before, this is needed when a moving interface comes close to an obstacle and jumps instantaneously to ensure that the jump is pointwise down in time and preserve the sub-solution property at jump points.

\subsection{Dynamics on the Real Line II}

We can now use the slowness of the obstacle-driven evolution in comparison to the kink/anti-kink attraction to show that the pinning constraint has no influence on the motion of a single kink/anti-kink pair.

\begin{theorem}[Asymptotically flat crystal]\label{theorem step line}
Let $x_{i,\eps} = i\,d_\eps$ for $d_\eps\gg \sqrt{\eps}$. Then there exist $\ul u_\eps \leq \ol u_\eps$ which are a viscosity sub- and super-solution of
\[
\begin{pde}
\frac{1}{|\log\eps|} \eps\, u_t &= \frac{1}{|\log\eps|}\left(\A u - \frac1\eps\,W'(u)\right) &\text{in }\R\setminus \bigcup_{i\in \Z}\overline{B_\eps(id_\eps)}\\
	u& =0 &\text{on }\bigcup_{i\in \Z}\overline{B_\eps(id_\eps)}
\end{pde}
\]
respectively such that
\[
\lim_{\eps\to 0} \ul u_\eps(t, \cdot) = \lim_{\eps\to 0} \ol u_\eps(t,\cdot) = \chi_{[-r(t), r(t)]}
\]
in $L^2(\R)$ for all $t>0$ with
\[
r(t) = \sqrt{r(0)^2 - \frac{t}{\int_{-\infty}^\infty (\phi')^2\dt}\,}.
\]
\end{theorem}

\begin{proof}
Choose $\delta>0$. Following \cite{Patrizi:2016aa} we know that for all small enough $\eps>0$ 
\[
\left[\phi\left(\frac{x + \ol x_\delta}\eps\right) - \eps \psi \left(\frac{x + \ol x_\delta}\eps\right)\right] +\left[\phi\left(\frac{\ol x_\delta -x}\eps\right) - \eps \psi \left(\frac{\ol x_\delta -x}\eps\right)\right] - 1
\]
is a sub-solution of the unpinned equation
\[
\eps\, u_t = \A u - \frac1\eps\,W'(u)
\]
when we choose $\ol x_\delta$ as the solution of the ordinary differential equation
\[
\dot{\ol x}_\delta = \frac1{\int_{-\infty}^\infty\left(\phi'\right)^2\dx} \,\frac{-1}{2\ol x_\delta} -\delta,\qquad \ol x_\delta(0) = r(0) - \delta.
\]
We just sketch the modifications which we need to make in the previous proof to apply it in this situation. Again, we can modify the interface choosing
\[
\eps^{-1/2} = \eps^{-1}\,\eps^{1/2} \ll L_\eps \ll \eps^{-1}\,d_\eps.
\]
This time, we need to modify the function $\phi_\eps = \phi - \eps\psi$. Abbreviate again $w_\eps =  (f_L\circ\phi_\eps) - \phi_\eps$ and compute
\[
w_\eps'' = [(f_L'\circ\phi_\eps) - 1]\,\phi_\eps'' + (f_L''\circ\phi_\eps)\,(\phi'_\eps)^2 \leq C\left(\frac1{L^3} + \frac{\eps}{L^2}\right)
 \]
 so
\begin{align*}
\A w_\eps (x) \leq 2L\,\left|\left| \,w_\eps''\right|\right|_{L^\infty(x-L,x+L)} + \frac2L\,\left|\left|\, w_\eps \right|\right|_{L^\infty(\R)} \leq C\,\left(\frac{1}{L^2} + \frac{\eps}{L}\right).
\end{align*}

The contribution to the attraction thus is $O(L^{-2} + \eps\,L^{-1}) = O(L_\eps^{-2})$, which was seen to be slow compared to the  kink/anti-kink attraction in the previous proof. When constructing sub-solutions in this setting, we only have to jump over obstacles when we come $L_\eps$-close (as before), but the obstacles are $d_\eps/\eps$-far apart, which is significantly further by our choice of $L_\eps$. Thus both the additional attraction and the fast motion close to obstacles disappear in the limit $\eps\to 0$. Thus the sub-solution converges to
\[
\chi_{[-\ol x_\delta, \ol x_\delta]}
\]
strongly in $L^1(\R)$. Now it suffices to take $\delta\to 0$. Super-solutions are obtained similarly.
\end{proof}

It is expected that the Theorem results can be extended to the case where several up and down steps occur by combining our methods with those of \cite{MR3338445,Patrizi:2016aa}.

We see that motion becomes slow also in this time scale as the compact step becomes wider and wider. If we take a limit such that one transition remains fixed at the origin and let the other one go to $\pm\infty$, we partially recover the statement of the previous Theorem as we see that in this time-scaling, the evolution of a single step is stationary. To recover the optimal time-scale, we could couple the initial width $r(0) = r_\eps(0)$ of the step to $\eps$.

\subsection{Periodic dynamics}

On a circle of finite radius, there is no analogue of a single step. Instead, we can consider the situation in which $\{u_\eps \approx 1\}$ is the majority phase. Without pinning, the majority phase takes over the minority phase in logarithmic time in a gradient flow. This happens precisely as it would if $\{u_\eps\approx 1\}$ is the minority phase and the pinning has no effect, just as on the real line. If $\{u_\eps \approx 1\}$ however is the majority phase, this would increase the energy, and the evolution becomes stationary on all timescales.

The use of energy methods relies on an analogue of Theorem \ref{theorem garroni mueller} being valid in one dimension. We formulate it at the end of this section and assume its validity throughout.

\begin{theorem}\label{theorem step circle}
Denote by $S^1_R = [-R/2, R/2]$ the circle of radius $R$ and $x_{i,\eps}$ be points on $S^1_R$ such that
\[
\min_{i\neq j}|x_{i,\eps}-x_{j,\eps}| \geq d_\eps \gg \sqrt{\eps}
\]
The number of points is denoted by $N_\eps$.
\begin{enumerate}
\item Let $r<R/2$. There exists a weak solution $u_\eps$ of 
\begin{equation}\label{eq evolution 2}
\begin{pde}
\frac{1}{|\log\eps|} \eps \,u_t &= \frac{1}{|\log\eps|}\left(\A u - \frac1\eps\,W'(u)\right) &\text{in }S^1_R\setminus \bigcup_{i=1}^{N_\eps}\overline{B_\eps(x_{i,\eps})}\\
	u &=0 &\text{on }\bigcup_{i =1}^{N_\eps} \overline{B_\eps(x_{i,\eps})}
\end{pde}
\end{equation}
such that $u_\eps(0,\cdot) \to \chi_{[-r/2,r/2]}$ which satisfies
\[
u_\eps(t,\cdot) \to \chi_{[-r(t), r(t)]}
\]
for all $t>0$ independently of the distribution of points $x_{i,\eps}$. Here $r(t)$ solves
\[
\dot r = - \frac1{2r} + 2 \sum_{n=1}^\infty \frac{2r}{(nR)^2 - 4r^2}, \qquad r(0)  = r/2.
\]
 If the points $x_{i,\eps}$ satisfy the conditions of Theorem \ref{theorem garroni mueller}/Proposition \ref{conjecture GM 1d}, then also $\E_\eps(u_\eps(0,\cdot))\to \E(\chi_{[-r/2,r/2]})$.

\item Let $r>R/2$ and assume that the points $x_{i,\eps}$ satisfy the conditions of Proposition \ref{conjecture GM 1d}. Then there exists a weak solution $u_\eps$ of 
\[
\begin{cases}
c_\eps \eps\, u_t = \frac{1}{|\log\eps|}\left(\A u - \frac1\eps\,W'(u)\right) &\text{in }S^1_R\setminus \bigcup_{i=1}^{N_\eps}\overline{B_\eps(x_{i,\eps})}\\
	\hspace{1.15cm}u =0 &\text{on }\bigcup_{i=1}^{N_\eps}\overline{B_\eps(x_{i,\eps})}
\end{cases}
\]
such that $u_\eps(0,\cdot) \to \chi_{[-r/2,r/2]}$ and $\E_\eps(u_\eps(0,\cdot))\to \E(\chi_{[-r/2,r/2]})$ which satisfies
\[
u_\eps(t,\cdot)\to \chi_{[-r/2,r/2]}
\]
for all $t>0$, independently of $c_\eps\to 0$. 
\end{enumerate}
\end{theorem}

\begin{proof} {\bf Proof of (1).} Note that the results of \cite{MR3338445} also hold in this setting and that an unpinned solution to the time-rescaled gradient flow is governed by this ODE. In the proof, the discussion of the constants in front of $\eps$-powers are finite is slightly more involved than in the case of finitely many layers. One has to use periodicity in an essential way always combining the force exerted by a couple of a kink/anti-kink pair to obtain cancellations between otherwise infinite forces.

We construct sub- and super-solutions periodically on $\R$ and then take them as functions on the circle. Note that the series
\[
U(x):= \sum_{k\in \Z} \left[ (\phi-\eps\,\psi)\left(\frac{x- kR - r/2}\eps\right) - (\phi-\eps\psi)\left(\frac{x- kR + r/2}\eps\right) \right]
\]
converges absolutely and uniformly for all $x\in\R$ by comparison with $\sum_{k=1}^\infty k^{-2}$ since for $kR> x+r/2$ we have
\begin{align*}
\phi\left(\frac{x- kR - r/2}\eps\right) - \phi\left(\frac{x- kR + r/2}\eps\right) &= 1- \frac{1}{W''(0)\,\frac{x- kR - r/2}\eps} + O\left(\frac{\eps^2}{(x- kR - r/2)^2}\right)\\
	&\quad - \left(1 - \frac{1}{W''(0)\,\frac{x- kR - r/2}\eps} + O\left(\frac{\eps^2}{(x- kR - r/2)^2}\right)\right)\\
	&= \frac1{W''(0)}\left(\frac\eps{x-kR-r/2} - \frac\eps{x-kR+r/2}\right) + O(\eps^2\,(kR)^{-2})\\
	&= \frac{\eps}{W''(0)} \frac{r}{(x-kR)^2 - (r/2)^2}+ O(\eps^2\,(kR)^{-2}).
\end{align*}
A similar estimate holds for $kR< x-r/2$ and for $\psi$. Hence, the partial sums of the series converge and a continuous limit exists. Since super-solutions for a finite number of kink/anti-kink pairs have precisely this form, we can construct a limiting (viscosity) super-solution to the unpinned problem using this series. Now, it is easy to see that again
\[
\ol u(x) = \min\left\{U(x), \ol u_l\right\}
\]
is a super-solution to the pinned equation. 

For weak sub-solutions, of course, we do not need this machinery, since the functions can easily be periodically extended as they become constant away from the interface. The proof proceeds like that of Theorem \ref{theorem step line} with an additional term in the calculations from periodicity. The resulting ODE is computed by periodicity:
\begin{align*}
\dot r &=\left(\sum_{n\in \Z\setminus\{0\}} \frac{1}{r - (r+nR)} - \frac1{r- (-r +nR)}\right) - \frac1{2r}\\
	&= - \frac1{2r} + 2 \sum_{n=1}^\infty \frac{2r}{(nR)^2 - 4r^2}.
\end{align*}

The initial condition with converging energies which lies above the sub-solution is given by
\[
u_0 = \min \left\{\sum_{k\in\Z}\left[\phi\left(\frac{x - kl - r/2 - \delta_\eps}\eps\right) - \phi\left(\frac{x -kl + r/2 + \delta_\eps}\eps\right)+1\right],\ubar\left(\frac x\eps\right)\right\}
\]
where $\ubar$ is the periodic solution of Lemma \ref{lemma obstacle sub-solution} with $M\equiv 0$. Furthermore, $\delta_\eps$ is a small parameter which ensures that the initial condition lies above the sub-solution.

{\bf Proof of (2).} Fix $t>0$. For any sequence $c_\eps>0$ and solutions $u_\eps$ to the evolution equation, we observe that $u_\eps(t,\cdot) \to \chi_{E(t)}$ up to a subsequence since the initial energies are bounded, and the energy decreases along the gradient flow.

Assume there is $c_\eps\to 0$ such that $u_\eps(t, \cdot)\to \chi_{E(t)}$ with $E(t) \neq [-r/2, r/2]$. Since the unpinned evolution equation wants to expand the $\{u=1\}$ phase under these initial conditions, we can easily construct a stationary sub-solution to the initial condition: the kink/anti-kink helps us, and the contracting force of the obstacles becomes negligible. Thus $[-r, r]\subset E(t)$ for all $t>0$.

We assume that an analogue of Theorem \ref{theorem garroni mueller} holds in one dimension. For a contradiction, assume that $[-r,r] \subsetneq E(t) \subsetneq S_R^1$. Then $\E(\chi_{E(t)}) \geq 2 + \Lambda\,\L^1(E(t)) > \E(\chi_{[-r/2,r/2]}$, which is a contradiction. The inequality holds since any set which is neither empty nor the whole circle has a perimeter $\geq 2$ in one dimension and since the $\{u=1\}$ phase was assumed to be expanding. 

If $E(t) = S^1_R$, then we use the fact that $u_\eps \in C^0([0,T],L^2)$ evolves continuously when considered as an $L^2$-valued function. This allows us to choose a different sequence $\tilde c_\eps$ such that the integral of $u_\eps$ at time $t$ is always strictly bounded away from both $r$ and $R$. Thus we have reduced this case to the previous one and obtain a contradiction like before.
\end{proof}

\begin{remark}
 It is an open question whether the statement above is stable in the sense that all solutions to \eqref{eq evolution 2} with initial conditions $u_\eps^0$ satisfying
\[
u_\eps^0\to \chi_{[-r,r]}, \qquad \E_\eps(u_\eps^0)\to \E(\chi_{[-r,r]})
\]
behave in the same way.
\end{remark}

Finally, let us state the result needed for the use of energy methods above.

\begin{proposition}\label{conjecture GM 1d}
Let $x_{i,\eps}\in S^1$ be points such that $1\leq i\leq N_\eps$ with $N_\eps/|\log\eps|\to \Lambda$ satisfying the following assumptions:

\begin{enumerate}
\item {\rm (well-seperated)} There exists $\beta<1$ independent of $\eps>0$ such that $d(x_{i,\eps},x_{j,\eps})>\eps^\beta$ for all $1\leq i\neq j\leq N_\eps$.

\item {\rm (finite capacity density)} The obstacles approach a multiple of the Lebesgue measure through $\frac{1}{|\log\eps|}\sum_{i=1}^{N_\eps}\delta_{x_i}\to \Lambda\,\L^2$ for $\Lambda\in (0,\infty)$.
\end{enumerate}

Take the space
\[
X_\eps:= \{u_\eps \in H^{1/2}(S^1)\:|\:u_\eps\equiv 0\text{ on }B_\eps(x_{i,\eps})\text{ for } 1\leq i\leq N_\eps\}
\]
and the energy functional
\[
\E_\eps:X_\eps\to\R, \qquad \E_\eps(u_\eps) =  \frac1{|\log\eps|}\left(\,[u_\eps]^2_{1/2} + \int_{S^1} \frac1\eps W(u_\eps)\dx\right)
\]
where $W$ is a periodic multi-well potential and $W\geq c\,\dist^2(\cdot,\Z)$ for some $c>0$. Then
\[
\left[\Gamma(L^2)-\lim_{\eps\to 0} \E_\eps\right](u) = \int_{S^1} \alpha(u)\dx + 4\int_{J_u}[u]\d\H^{1}
\]
where $u\in BV(S^1,\Z)$, $[u] = u^+-u^-$ denotes the jump of $u$ on the jump set $J_u$ and $\alpha(z)$ is determined as the solution of the cell problem
\[
\alpha(z) = \inf\left\{ \frac12\,[w]_{1/2,\R^2}^2 + \int_{\R}W(w)\dx \:\bigg|\: w-z\in H^{1/2}(\R), \:w\equiv 0 \text{ on }B_1(0)\right\}.
\]
\end{proposition}

We will not prove the proposition in this article, the sceptical reader may also take it as a conjecture. For a useful compactness property, either a more restrictive distribution of obstacles or a potential growing at $\infty$ should be imposed.

\subsection{External Driving Forces}

Let us assume that an external sheer force is applied to the crystal. On the scale where the crystal can be assumed to be periodic, an applied force is constant in space and thus enters the evolution equation as an additive constant.

\begin{theorem}\label{theorem external circle}
Denote by $S^1_R = [-R/2, R/2]$ the circle of radius $R$ and $x_{i,\eps}$ be points on $S^1_R$ satisfying the conditions of Proposition \ref{conjecture GM 1d}. Let $r<R$, $f\in\R$. There exists a weak solution $u_\eps$ of 
\begin{equation}\label{eq evolution force}
\begin{cases}
\eps \,u_t = \frac{1}{|\log\eps|}\left(\A u - \frac1\eps\,W'(u)\right)  + f&\text{in }\R\setminus \bigcup_{i=1}^{N_\eps}\overline{B_\eps(x_{i,\eps})}\\
	\hspace{1.15cm}u =0 &\text{on }\bigcup_{i =1}^{N_\eps} \overline{B_\eps(x_{i,\eps})}
\end{cases}
\end{equation}
with an initial condition satisfying 
\[
u_\eps(0,\cdot) \to \chi_{[-r/2,r/2]}\quad \text{in }L^2(S^1), \qquad \E_\eps(u_\eps(0))\to \E(\chi_{[-r/2,r/2]})
\]
such that the following hold:

\begin{enumerate}
\item If $f<0$, then $u_\eps(t, \cdot) \to \chi_{[-r/2 + |f|t, r/2 -|f|t]}$ in $L^2(S^1)$ for all $t>0$ (the characteristic function of the empty set being zero).

\item There exists $f_0>0$ such that for $0<f<f_0$, we have $u_\eps(t, \cdot) \to \chi_{[-r/2,r/2]}$ in $L^2(S^1)$ for all $t>0$. This also holds if we accelerate the solutions to any faster time-scale.
\end{enumerate}
\end{theorem}

The proof proceeds exactly as before, but the applied force now acts on the fast time-scale when it is contracting so that the kink/anti-kink attraction disappears in the limit. In the other direction, we note that for small forces, an energy barrier still needs to be overcome. Similar results could be obtained if the force $f$ is allowed to scale with $\eps$ -- if for example $f_\eps \sim |\log\eps|^{-1}$ is negative, then the attraction of interfaces and the external force are assumed to act additively on the same time scale, compare \cite{gonzalez2010slow} where periodic forcing is considered.

\section{Two-dimensional Dynamics}\label{section 2d}

\subsection{Technical Points}

The one-dimensional evolution reaches the macroscopic time-scale at an obstacle distance of $d_\eps \sim \sqrt{\eps/|\log\eps|}$ which is the natural distance of obstacles in the setting of Theorem \ref{theorem garroni mueller}. In two dimensions, this is still expected to be slow.

Let us for the moment assume that the obstacles are distributed on a perfect grid $d_\eps\cdot \Z^2$ in the plane. Then we can use the moving interface sub-solutions constructed in one dimension as a sub-solution by extending them as constant in the second direction. They are still pinned sub-solutions since at $\{u=0\}$, the sub-solution property holds trivially on the non-pinned set. However, these extended sub-solutions can be considered as sub-solutions in the situation when the obstacles are $\eps$-tubes around lines rather than unions of $\eps$-balls. The volume of the pinning set is $N_\eps\cdot \eps^2 \sim \eps\,|\log\eps|$, while the volume of the $\eps$-tubes is proportional to $\sqrt{N_\eps}\cdot \eps \sim \sqrt{\eps\,|\log\eps|}$, so considerably larger. While the volume of the pinning set is a bad proxy for estimating its influence, this simple observation suggests that using a one-dimensional construction may well over-estimate the influence of pinning. Indeed, a back-of-the-envelope calculation like in Section \ref{section heuristic} gives much slower speed for super-solutions.

The modification of the interface needs to be done more carefully in this setting, since the flattening out of the interface to facilitate glueing induced motion on the macroscopic time-scale in this scaling. The refined modification is presented below.

\begin{lemma}\label{lemma sub-solution 2d}
Let $1/2 < \zeta \leq 1$, $F>0$, $l\gg1$. Then there exists 
\[
\ubar\:\in\: C^{1/2}\left(\R^2\right) \,\cap\, C^{1,1/2}_{loc}\left(B_2\setminus \overline{B_1}\right) \,\cap\, C^{1,1/2}\left(\R^2\setminus B_{3/2}\right)
\]
 with the following properties:

\begin{enumerate}
\item We have
\[
\ubar\equiv 0 \text{ on }B_1\qquad\text{and}\quad \A \ubar-W'(\ubar) \geq \frac{1}l + \frac{F}{l^{1+\zeta}} \text{ on  }\R^2\setminus B_1.
\]

\item The function $\ubar$ is constant on $\R^2\setminus B_{l^\zeta/3}$ and
\[
\left|\lim_{|x|\to\infty}u(x) - \left(1- \frac{1}{W''(0)\,l}\right)\right| \leq \frac{C}{l^{2\zeta}}
\]

\item The growth estimate 
\[
\ubar(x) \geq 1- \frac{1}{W''(0)\,l} - \frac{c}{|x|^2}
\]
holds for some $c>0$.
\end{enumerate}
All constants are independent of $l$.
\end{lemma}

\begin{proof}
Take the sub-solution $\ubar$ constructed in Lemma \ref{lemma obstacle sub-solution} on the circle of length $l^\zeta$ with 
\[
M = M_l = \frac{1}{l^{1-\zeta}} + \frac{F}{l}
\]
and extend it to the whole real line as in Lemma \ref{extension lemma}. Now set
\[
u_{2D}(x) = \ubar(|x|).
\]
Since the norm is $1$-Lipschitz, the function $u_{2D}$ is $C^{0,1/2}$-H\"older continuous and since the norm is $C^\infty$-smooth away from the origin, $u_{2D}$ is exactly as smooth as $\ubar$ outside $\overline{B_1(0)}$. 

The function is constant on $\R^2\setminus B_{l^\zeta/2}(0)$ and satisfies the well-known growth estimate. Note that $F$ vanishes in the error estimate to leading order since $l^{1+\gamma}\ll l^{2\gamma}$. The super-solution property is established by comparing the rotationally symmetric extension to a non-radial extension. Assume that $x= |x|\cdot e_1$ and observe that
\begin{align*}
\A u_{2D}(x) &= \int_{\R^2}\frac{\ubar(|y|) - \ubar(|x|)}{|y-x|^3}\dy \\
	& \geq \int_{\R^2}\frac{\ubar(y_1) - \ubar(x_1)}{|y-x|^3}\dy\\
	&= c_{2,1} \int_\R \frac{u(y_1) - u(x_1)}{|y_1-x_1|^2}\d y_1\\
	&= c_{2,1} \A\ubar(|x|).
\end{align*}
since $\ubar$ is monotone growing away from the origin and $|x| = x_1$. The same holds after rotation for any point $x\in \R^2\setminus B_1(0)$. 

We used that for a function $f:\R^n\to\R$, $f(\hat x, x_n) = g(\hat x)$ for some $g:\R^{n-1}\to\R$ we have $\A^{\R^n}f(\hat x, x_n) = \A^{R^{n-1}}g(\hat x)$ when the fractional Laplacian is computed as a singular integral. The normalising constant $c_{2,1} \neq 1$ appears here because we neglected normalising the fractional Laplacian before. Since the same re-normalisation affects the half-Laplacian acting on the interface and the obstacle sub-solution in the same way, we will not make a difference here and remark only that the Lemma holds for the properly normalised operator.
\end{proof}

In two dimensions, the downward force exerted by the pinning constraint decays faster and the optimal transition approaches $+1$ more quickly. This is connected to the fact that small balls shrink faster in two dimensions, or equivalently, that the boundary condition at infinity has a stronger upwards pull since large circles have increasing measure while two points in one dimension always have the same mass.

\begin{lemma}\label{lemma minimiser 2d}
Let $u\in 1+ H^{1/2}(\R^2)$ be a minimiser of 
\[
\E(u) = [u]_{1/2}^2 + \int_{\R^2} W(u)\dx
\]
under the constraint $u\equiv 0$ on $B_R(0)$. Then $u \in C^{1/2}(\R^2)$ is radially symmetric, smooth away from the pinning set, and satisfies $1 - \frac{c}{|x|^3} \leq u(x) < 1$ for all $x\in \R^2$ and some $c\geq 1$.
\end{lemma}

The proof is a slight variation of that of Lemma \ref{lemma obstacle sub-solution} or Lemma \ref{lemma subsolution line}.

\subsection{Dynamics in the Plane}

We will now prove that the pinning constraint acts on a much slower time-scale than the kink/anti-kink attraction by considering the model problem of a single infinitely long straight interface on $\R^2$ perfectly aligned with the grid. 

\begin{theorem}\label{theorem 2d slow}
Let $\Gamma_\eps = d_\eps\cdot \Z^2$ for $d_\eps\gg \eps$. Then there exists $\ul u_\eps \leq \ol u_\eps$ which are a viscosity sub- and super-solution of
\begin{equation*}
\begin{pde}
c_\eps^\pm \eps \,u_t &= \frac{1}{|\log\eps|}\left(\A u - \frac1\eps\,W'(u)\right) &\text{in }\R^2\setminus \bigcup_{i\in \Z}\overline{B_\eps(id_\eps)}\\
	u& =0 &\text{on }\bigcup_{i\in \Z}\overline{B_\eps(id_\eps)}
\end{pde}
\end{equation*}
respectively. When we choose $c_\eps^+ = \frac{\eps^2} {d_\eps^3\,|\log\eps|}$ for the time scaling of the super-solution and $c_\eps^- = |\log\eps|\,c_\eps^+$, there are constants $c, C>0$ such that
\[
\lim_{\eps\to 0} \ul u_\eps(t, \cdot) = \chi_{[ct,\infty)\times \R}, \qquad \lim_{\eps\to 0}\ol u_\eps(t,\cdot) = \chi_{[Ct,\infty)\times \R}
\]
in $L^2_{loc}(\R)$ for all $t>0$. In particular, the gradient flow is slow of some order between
\[
\eps^{1/2}\,|\log\eps|^{1/2} \leq c_\eps \leq \eps^{1/2}\,|\log\eps|^{3/2}
\]
in the line-tension scaling $d_\eps \sim \eps^{1/2}|\log\eps|^{-1/2}$.
\end{theorem}

Here, we miss the optimal order by a logarithmic term as the sub-solution moves on a faster time-scale than the super-solution. This discrepancy is due to our use of a radially extended function rather than a fully two-dimensional construction. The two-dimensional growth rate is observed in Lemma \ref{lemma minimiser 2d}, and we expect the super-solution to give the right order of movement rather than the sub-solution -- see Remark \ref{remark decay 2D} for the difficulties related to constructing sub-solutions directly in two dimensions.

\begin{proof}[Proof of Theorem \ref{theorem 2d slow}]
Like in the proof of Theorem \ref{theorem slowness line}, we begin by constructing sub-solutions in a blow-up scale.

{\bf First modification.}
Like in the one-dimensional case, denote $l = l_\eps = d_\eps/(\eps N)$ for suitably large $N\in\N$. Take $f_L$ like in Lemma \ref{lemma modified interface}, but this time for $L = l^{3/2}$. Choose $g\in C^\infty(\R)$ such that
\[
0\leq g \leq 1, \qquad g(t) = \begin{cases}1 &|t|\geq 2l\\ 0 &|t|\leq l\end{cases}, \qquad |g'|\leq \frac{2}l, \qquad |g''|\leq \frac4{l^2}
\]
and $c_1, c_2>0$ to be specified later. Set 
\[
\wt \phi(x) = f_L\circ\phi(x_1) - \frac{c_1\,\log(l)}{l^2}\,g(x_1) - \frac{c_2\,\log(l)}{l^3}
\]
and compute
\begin{align*}
\left|\A (\wt \phi -\phi)(x)\right| &\leq \left|\A(f_L\circ\phi - \phi)(x)\right| + \left|\frac{c_1\,\log l}{l^2}\,\A g(x)\right|\\ 
	&\leq \frac{C}{L^2} + \frac{c_1\log(l)}{l^2}\left(\int_{B_l(x)}\frac{||\,D^2g||}{|y-x|}\dy + \int_{\R^2\setminus B_l(x)}\frac{2\,||g||}{|x-y|^3}\dy\right)\\
	&\leq \frac{C}{l^3} + \frac{c_1\log(l)}{l^2}\left(\frac{4}{l^2}\int_0^l\frac1r\,(2\pi r)\dr + 2\int_l^\infty \frac1{r^3}\,(2\pi r)\dr\right)\\
	&\leq \frac{C\,\log(l)}{l^3}.
\end{align*}
like in Lemma \ref{lemma modified interface} because $L^2 = l^3$. Thus $\wt \phi$ satisfies
\begin{align*}
\A\wt \phi - W'(\wt\phi) &= \A(\wt\phi - \phi) + \A \phi - W'(\phi) + \left[W'(\phi) - W'(\wt\phi)\right]\\
	&= \A(\wt\phi - \phi)+ \left[W'(\phi) - W'(\wt\phi)\right]\\
	&\geq \begin{cases} \frac{- c_Wc_2\,\log l}{l^3} & \wt\phi(x) \in [\beta,1-\beta]\\ 
		\frac{\tilde c_W\,c_2\,\log(l)}{l^3} & \wt\phi(x)\in (0,\beta] \text{ or } \wt\phi(x)\in [1-\beta, 1]\\
		\frac{\tilde c_W\,c_1\,\log(l)}{l^2}  & |x|\geq 2l.
	  \end{cases}
\end{align*}
All constants are positive and depend only on $W''$. 

{\bf Second modification.} Let $y_{i,j,\eps} = Nl\,(i,j)$ be an enumeration of the pinning sites after rescaling with $i,j\in\Z$. For pinning sites with $Nli \leq l^{3/2}+l$, we need an additional modification to flatten the interface before we can insert obstacles. 

Choose $1/2 < \gamma < 1$ and a bump-function $\eta$ such that
\[
0\leq \eta\leq 1, \qquad\eta(x) = \begin{cases}1& |x|\leq l^\gamma\\ 0 &|x|\geq 2l^\gamma \end{cases}, \qquad |\nabla \eta|\leq \frac2{l^\gamma}, \qquad|D^2\eta|\leq \frac4{l^{2\gamma}}
\]
and set
\[
\tilde u_\lambda (x) = \left(1 - \sum_{|i| \leq \frac{2\,l^{1/2}}N} \:\sum_{j\in\Z} \eta(x - y_{i,j,\eps})\right)\,\wt\phi(x_1 + \lambda l) + \sum_{|i|\leq 2 \sqrt{l}/N,\: j\in\Z} U_{i,\eps}\,\eta(x - y_{i,j,\eps})
\]
with
\[
U_{i,\eps} = \wt\phi(\,(Ni -\lambda)l) + O(l^{\gamma-2})
\]
for a small term $O(l^{\gamma-2})$ to be chosen later. This is a function which mostly looks like (a translated version of) $\wt\phi$, but is flattened at the pinning sites. The parameter $\lambda$ will later be used for the time-evolution. Note that 
\[
|\wt\phi(x) - U_{i,\eps}| = \frac1{W''(0)} \left|\frac1{x} - \frac1{Nil-\lambda}\right| + O(l^{-2}) + O(l^{\gamma-2})  \leq C\, l^{\gamma -2}
\] 
where $\eta(x- y_{i,j,\eps}) \neq 0$, so when we set
\begin{align*}
\eta_{i,j}(x) &= \eta(x- y_{i,j,\eps}),\\
w &\coloneqq \tilde u_\lambda - \wt \phi\\
	&= \sum_{i,j} \left(U_{i,\eps} - \wt \phi (x_1 + \lambda l)\right)\,\eta_{i,j}
\end{align*}
 we observe that 
\begin{align*}
\nabla w (x) &= \sum_{i,j}\left(U_{i,\eps} - \wt \phi\right)\,\nabla \eta_{i,j} - \wt \phi'\, \eta_{i,j} \,e_1,\\
D^2 w(x) &= \sum_{i,j} \left(U_{i,\eps} - \wt \phi)\right)\,D^2 \eta_{i,j} - \wt \phi' \, (e_1\otimes \nabla \eta_{i,j} + \nabla \eta_{i,j} \otimes e_1) - \wt\phi''\, \eta_{i,j}\,e_1\otimes e_1\\
	&= O\left(l^{-(2+\gamma)}\right).
\end{align*}
Therefore, the usual argument shows that
\begin{align*}
\left|\A w(x)\right| &\leq  \int_{B_{l^\gamma}(x)}\frac{||D^2w||}{|x-y|}\dy + \int_{\R^2\setminus B_{l^\gamma}(x)}\frac{||w||}{|x-y|^3}\dy = O(l^{-2})
\end{align*}
on $\R^2$. This estimate can be improved if we are far from the next pinning site. Namely, assume that $\min_{i,j}|x-y_{i,j,\eps}|\geq 2l$, then the first term vanishes and the sharper estimate
\begin{align*}
\left|\A w(x)\right| &
	\leq \sum_{i,j} \frac{C}{|x- y_{i,j,\eps}|^3} \int_{B_{l^\gamma}(y_{i,j,\eps})} 2\,||w||\dy\\
	&\leq \sum_{i,j} \frac{C\,l^{2\gamma}\,l^{\gamma -2}}{|x-y_{i,j,\eps}|^3}\\
	&= O(l^{3\gamma - 5})
\end{align*}
holds. From now on, take $\gamma=2/3$, so that $\A w= O(l^{-3})$. Note that this holds true for all $x$ with $|x|\leq 2l$.

{\bf Inserting Obstacles: Right Half-Space.} We now insert the obstacle sub-solutions from Lemma \ref{lemma sub-solution 2d} into the flattened out sites and on the half-spaces that are flattened out by $f_L$. First we deal with the right hand side of the interface where $\phi$ is close to $1$. 

We concentrate on the obstacles in the flattened discs since the flattened half-plane can be treated similarly. The height of the obstacle at $y_{i,j,\eps}$ is $1- \frac{1}{W''(0)\,(Nil -\lambda)} - O((Nil)^{\gamma-2})$ for $j\leq 2l^{1/2}/N$ which is 
\[
1- \frac{1}{W''(0)\,l} - O(l^{-4/3}) = 1- \frac{1}{W''(0)\,(Nil)}  - \frac {F}{(Nil)^{4/3}} + o(l^{-4/3})
\]
for some bounded sequence $F = F_j\in\R$, so we choose $\zeta=1/3$ in Lemma \ref{lemma sub-solution 2d}. A lower order perturbation in either term gives a matching height between the interface and the inserted obstacle so that we can glue the obstacle into the modified domain. The sub-solution for an obstacle on $\R^2$ is constant for arguments $|x|\geq l^\zeta = l^{1/3}$ and the flattened out disc in the obstacle has a radius of $l^{2/3}$, so the glueing does not cause more problems. The resulting function is denoted by $u$.

Let us check that the sub-solution condition is still satisfied at the obstacles we just inserted. An ideal interface would exert a pressure of
\begin{align*}
\A\phi(Nil) &= W'(\phi(Nil))\\
	&= W''(0)\,(1-\phi(Nil)) + O((Nil)^{-2})\\
	&= W''(0) \cdot\frac{1}{W''(0)\,Nil} + O((Nil)^{-2})\\
	&= \frac1 {Nil} + O((Nil)^{-2})
\end{align*}
at $y_{i,j,\eps}$ which is compensated by the obstacle by construction. The same is true up to order $O(l^{\gamma-2}l^{-\gamma}) = O(l^{-2})$ for the modified interface $\tilde u_\lambda$, which is also compensated. The obstacle $y_{i', j', \eps}$ has a distance of $Nl\,\sqrt{(i-i')^2 + (j-j')^2}$ to $y_{i,j,\eps}$, so their contribution to the pressure is negligible. Namely, the pressure created by the obstacles at a point $x$ is
\begin{align*}
\sum_{i,j} \int_{B_l^{1/3}(y_{i,j,\eps})}\frac{u(y) - \tilde u_\lambda(y)}{|y -x|^3}\dy 
	&\leq \sum_{i,j} \left[\int_{B_1(y_{i,j,\eps}}\frac{1}{|y -x|^3}\dy + \int_{B_{l^{1/3}}(y_{i,j,\eps})\setminus B_1} \frac{\frac1{|y-y_{i,j,\eps}|^2}}{|y-x|^3}\dy\right]\\
	&= O\left(  \frac{1+ \log(l)}{\dist(x, Nl\Z)^3}\right)\\
	&= \begin{cases} O(l^{-2}\log(l)) &\text{if }\dist(x, Nl\Z)\geq l^{2/3}\\ O(l^{-3}\log(l)) & \text{if }\dist(x, Nl\Z)\geq l
	 \end{cases}
\end{align*}
which is compensated either by a sufficiently large constant $c_1$ close to the obstacles or by $c_2$ away from the obstacles. At the interface, it can be compensated by a speed $O(l^{-3}\log(l))$.

{\bf Inserting Obstacles: Left Half-Space.} Since the interface could only be modified for $|x|\geq L\geq l^{3/2}$, we also need to insert obstacles into the flattened out discs in the left half space in two dimensions. 

Being close to phase $\{u\approx 0\}$, the construction of a stationary obstacle sub-solution does not go through. 
Instead we can take a sub-solution of the periodic obstacle problem at phase $\{u\approx 1\}$ for $F =0$, an auxiliary double-well potential $\wt W$ and multiply it by a factor
\[
h_l \sim \frac{1}{W''(0)\,l}
\]
which allows for continuous glueing. Here, the interface pulls upwards with force $\sim \frac1l$, while the self-pressure of the obstacle is
\begin{align*}
\A (h_l u) - \wt W'(h_l u) &= h_l\,\A u - \wt W'(h_lu)\\
	&= h_l\wt W'(u) - \wt W'(h_lu)\\
	&\leq \frac{C_{\wt W}}{l}.
\end{align*}
When we choose $\wt W$ suitably, $C_{\wt W}$ can be made as small as we need. Thus the upwards pull of the interface compensates the self-pressure. 

{\bf Conclusion and Rescaling.} Overall, the calculations show that a function as constructed above is a sub-solution if the interface moves with a speed $\lambda = O(\log(l)\,l^{-3})$. Again, the suitable monotonicity of $\wt\phi$ and the precise construction ensure the sub-solution property at non-smooth times. When passing to the macroscopic scale as
\[
\ul u_\eps(t,x) = u\left(\frac{t}{\eps^2|\log\eps|\,c_\eps}, \frac x\eps\right)
\]
we observe that the interface moves a distance $d_\eps \sim \eps l_\eps$ in a time $t_\eps$ proportional to the product of the re-scaling factor with the quotient of the travelled distance $l_\eps$ in the blow-up scale and the speed $l_\eps^{-3}\,\log(l_\eps)$ in the blow-up scale, i.e.
\[
t_\eps \sim   \eps^2|\log\eps|\,c_\eps\cdot \frac{d_\eps}{\eps}\left(\frac{\eps^3\, |\log\eps|}{d_\eps^3}\right)^{-1} = \frac{c_\eps\,d_\eps^4}{\eps^2}.
\]
To obtain a uniform speed on the order $O(1)$ we choose $t_\eps\sim d_\eps$ or equivalently
\[
c_\eps^- = \frac{\eps^2}{d_\eps^3}.
\]

{\bf Super-solutions.} Super-solutions are constructed in analogy to the one-dimensional case. The growth rate $1- \frac{c}{|x|^3}$ from Lemma \ref{lemma minimiser 2d} leads to the fact that the logarithmic term in the integral is not present in the super-solution. A simple calculation shows that super-solutions move on the slower time-scale
\[
c_\eps^+ = \frac{\eps^2}{d_\eps^3\,|\log\eps|}.
\]
\end{proof}

\begin{remark} A similar argument can be made when the interface is not perfectly aligned with the grid. Take the grid $\Gamma_\eps = S(\phi)\cdot \left(d_\eps\cdot \Z^2\right)$ where $S(\phi)$ is the rotation matrix
\[
S(\phi) = \begin{pmatrix}\cos \phi &\sin \phi \\ - \sin\phi &\cos\phi\end{pmatrix}.
\]
If we take $\phi$ such that $\tan(\phi)\in \mathbb{Q}$, then the first component $z_1$ of $z\in \Gamma_\eps$ is
\[
z_1 = d_\eps\,(n\cos\phi + m \sin\phi) = d_\eps\cos\phi\,(n + m\tan\phi)
\]
for some $n,m\in\Z$. Since we assumed $\tan\phi$ to be rational, this is a discrete periodic subset of the real line and the distance between any two points is proportional to $d_\eps$. Also the fractional Laplacian can be estimated as before, so Theorem \ref{theorem 2d slow} also holds for rotated square grids. 

Equally well, we could rotate the interface instead of the grid. This resembles the settings of \cite{MR2231253,MR2463225} where a straight front in a periodic medium is considered for sharp interface mean curvature flow and for a local Allen-Cahn equation. Our setting differs in the use of a non-local differential operator and in that the obstacles are comparable to the size of the interface, but the distances between them lie on a much larger scale. 
\end{remark}

\subsection{Dynamics on a Torus}

Finally we state the main result as applied to the case of \cite{MR2231783}. Note that the inclusion of a constant force $f$ in the energy is a compact perturbation, so 
\[
\wt \E_\eps(u) \coloneqq \E_\eps(u) - \int_{\T^2}fu\dx \qquad\xrightarrow{\Gamma(L^2)}\qquad \E(u) - \int_{\T^2}fu\dx
\]
for any constant $f\in \R$. 

\begin{theorem}\label{theorem 2d torus}
Denote by $\T_a^2 = \R^2/\left(a\cdot \Z^2\right)$ the flat square torus with volume $A=a^2$. Consider the evolution equation
\[
\begin{pde}
c_\eps\,\eps u_t & = \frac1{|\log\eps|}\left(\A u - \frac1\eps\, W'(u)\right)  + f& \text{in }[0,\infty)\times \left[ \T^2_a\setminus \bigcup_{k=1}^{N_\eps} \overline{B_\eps(x_{i,\eps})}\right]\\
u &=0&\text{on }(0,\infty)\times \bigcup_{k=1}^{N_\eps} \overline{B_\eps(x_{i,\eps})}\\
u &= u_\eps^0&\text{at }t=0
\end{pde}
\]
where the pinning sites $x_{i,\eps}$ satisfy the assumptions of Theorem \ref{theorem garroni mueller} and additionally
\begin{enumerate}
\item the distribution assumption
\[
\{x_{1,\eps}, \dots, x_{N_\eps,\eps}\} \subset \bigcup_{j,k=1}^{\sqrt{\frac{|\log\eps|}{\eps}}} B_{r_\eps}\left(d_\eps\cdot(j,k)\right).
\]
for $r_\eps\ll d_\eps = \sqrt{\frac{\eps}{|\log\eps|}}$ (i.e.\ the pinning sites lie in small discs around grid points) and
\item the number of obstacles per disk is uniformly bounded:
\[
\# \left(\{x_{1,\eps}, \dots, x_{N_\eps,\eps}\} \cap B_{r_\eps}\left(d_\eps\cdot(j,k)\right)\right) \leq M
\]
for some $M\in\N$ independently of $\eps, j, k$.
\end{enumerate}
Then we find $u_\eps^0\to \chi_{[-r/2,r/2]\times[0,a]} =:u$ in $L^2(\T^2_a)$ (a strip of width $r$ around the torus)  such that $\E_\eps(u_\eps^0)\to \E(u)$ and the following hold:
\begin{enumerate}
\item[(i)] If $f =0$ and $r>a/2$, then $u_\eps(t,\cdot) \to u$ in $L^2(\T^2_a)$ independently of $c_\eps\to 0$.

\item[(ii)] If $f = 0$ and $r<a/2$, then $u_\eps(t,\cdot) \to \chi_{[-r(t), r(t)]\times[0,a]}$ in $L^2(\T^2_a)$ for $c_\eps = |\log\eps|^{-1}$ where
\[
\dot r = - \frac1{2r} + 2 \sum_{n=1}^\infty \frac{2r}{(nR)^2 - 4r^2}, \qquad r(0)  = r/2.
\]

\item[(iii)] If $0 < f < f_0$ for some $f_0$ depending only on the capacity $\alpha$ of dislocations (i.e.\ the potential $W$) and the limiting density $\Lambda = \lim_{\eps\to 0}\frac{\eps}{|\log\eps|}N_\eps\in (0,\infty)$, then $u_\eps(t,\cdot) \to u$ in $L^2(\T^2_a)$ independently of $c_\eps\to 0$. (This is valid also for $f_\eps>0$ if $f_\eps \gg |\log\eps|^{-1}$.)

\item[(iv)] If $f<0$, then $u_\eps(t,\cdot)\to \chi_{[-r/2 +|f|t, r/2 - |f|t]\times[0,a]}$ for $c_\eps \equiv 1$. 
\end{enumerate}
If $W\in C^4(\R)$ and $W^{(3)}(0) = 0$ (e.g.\ if $W' = \sin$), then we can generalise the distribution assumptions as follows.

\begin{enumerate}
\item[(1')] There exist $d_\eps'\gg \eps^{2/3} |\log\eps|^{1/3}$ and $r_\eps \ll d_\eps'$ such that 
\[
\{x_{1,\eps}, \dots, x_{N_\eps,\eps}\} \subset \bigcup_{j,k=1}^{1/d_\eps'} B_{r_\eps}\left(d_\eps'\cdot(j,k)\right).
\]
\item[(2')] the number of obstacles per disk is uniformly bounded:
\[
\# \left(\{x_{1,\eps}, \dots, x_{N_\eps,\eps}\} \cap B_{r_\eps}\left(d_\eps'\cdot(j,k)\right)\right) \leq M
\]
for some $M\in\N$ independently of $\eps, j, k$.
\end{enumerate}
\end{theorem}

The proof is a combination of the analogue statement in one dimension and the more subtle modification of the interface in two dimensions described above with a few additional facets:

\begin{enumerate}
\item Note that 
\[
|\psi'(x)|\leq \frac{C}{1+|x|^2} \qquad \Ra \qquad |\psi(x)|\leq \frac{C}{1+ |x|},
\]
so at the nearest obstacle where we need to modify we use $l_\eps \sim d_\eps/\eps \sim (\eps\,|\log\eps|)^{-1/2}$ to calculate
\[
\eps \psi\left(l_\eps\right) = O(\eps^{3/2}\,|\log\eps|^{1/2}) \ll \left(\eps\,|\log\eps|\right)^{3/2} = l_\eps^{3}.
\]
This allows us to carry out the same modifications as before without paying much attention to the corrector, which is a lower order perturbation only at the closest pinning site. We have enough wiggle room to come closer to the pinning sites and jump shorter by a logarithmic term, so again we can argue that neither the contracting force nor the jumps matter in the limit.

If $W^{(3)}(0) = 0$, then we can show that $\psi$ decays as $x^{-2}$ at $\pm\infty$, not just as $x^{-1}$, thus we can come closer to the corrected interface with the obstacles without having to take care of bigger complications in the modification process. 

If we could improve the order at which the sub-solution moves to the order of the super-solution, it would suffice to require $d_\eps\gg \eps^{2/3}$, which is the order at which the bulk term induces logarithmically fast motion. 

\item When we denote by $u$ the solution to the cell-problem from Lemma \ref{lemma minimiser 2d} and by $x_{i,\eps}$ the pinning sites, we see that the initial condition
\[
u_\eps^0(x) = \min \left\{ \phi\left(\frac{x_1 + r/2}\eps\right), \phi\left(\frac{-x_1 - r/2}\eps\right), u\left(\frac{x-x_{i,\eps}}\eps\right)\right\}_{1\leq i \leq N_\eps}
\]
is trapped between the sub- and super-solution constructed before. Since all three components have converging energies, also $\E_\eps(u_\eps^0) \to \E(u^0)$.
\end{enumerate}

The first statement covers the case of obstacles located on a square grid, the second case allows for relatively general arrangements in a denser grid with many vacancies.

In particular, we see that in none of the four cases above we obtain the gradient flow of the the limiting energy as limit of the evolutions, which behaves as follows:

\begin{enumerate}
\item If $f=0$, the $\{u=1\}$-phase contracts with constant velocity stemming from the bulk energy term.
\item If $f<0$, the $\{u=1\}$-phase contracts with constant velocity stemming from both the bulk-energy term and the external force. Here, the behaviour is correct, but the velocity is governed only by the external force.
\item If $0<f<f_0$, the $\{u=1\}$-phase contracts with constant velocity stemming from the bulk-energy which dominates the small external force in the opposite direction. Here, in fact a small force $f_\eps$ already suffices to cause qualitatively wrong behaviour.
\end{enumerate}

\begin{remark}
Similarly as for the propagation of a front on the whole space, we can also have tilted grids. On a torus, we obtain a tilted grid by labelling equidistant points on the edges of a square and then connecting $i d_\eps$ on the bottom with $(i+k)d_\eps$ on the top (and periodically extended). The same result holds then, up to slight technical complications.
\end{remark}

Finally, we apply the Theorem to show that the limit of the pure gradient flows without external force in the usual fast time scale is not mean curvature flow. 

\begin{corollary}\label{corollary 2d general}
Under the same assumptions as Theorem \ref{theorem 2d torus} and the assumption that $\Lambda>\Lambda_0>0$ for a suitable $\Lambda_0$, there exists a sequence of initial conditions $u_\eps^0\to u = 1- \chi_{B_r(0)}$ for some small $r>0$ such that $\E_\eps(u_\eps)\to \E(u)$ and of solutions $u_\eps$ of
\[
\begin{cases}
\epsilon \,\partial_t u_\eps = \frac1{|\log\eps|}\left(\A u_\eps - \frac1\eps\,W'(u_\eps)\right) &\text{in }[0,\infty)\times \left[ \T^2\setminus \bigcup_{k=1}^{N_\eps} \overline{B_\eps(x_{i,\eps})}\right]\\
u=0&\text{on }(0,\infty)\times \bigcup_{k=1}^{N_\eps} \overline{B_\eps(x_{i,\eps})}\\
u= u_\eps^0&\text{at }t=0
\end{cases}
\]
such that $u_\eps(t,\cdot) \to 1- \chi_{E(t)}$ for some set $E(t)$ for all times $t$, but the boundaries $\partial E(t)$ are not moving by either mean curvature flow or the gradient flow of $\E$.
\end{corollary}


\begin{proof}
At small circles, the line energy dominates the bulk term in both the energy $\E$ and its gradient flow, while circles of radius $r> \frac{1}{\Lambda\alpha}$ are expanding. Assume that $\Lambda_0$ is so large that $\partial B_{2/(\Lambda\alpha)}(0)$ is a round circle on the torus. Then choose $r \in (\,(\Lambda\alpha)^{-1}, 2(\Lambda\alpha)^{-1})$. Both the gradient flow of $\E$ and mean curvature flow of $\partial B_r(0)$ exist smoothly up to some small positive time. 

For energetic reasons, the initial set cannot shrink, so $B_r(0)\subset E(t)$ and $E$ does not evolve by mean curvature flow. Using straight interfaces as barriers, we use Theorem \ref{theorem 2d torus} to show that $E$ cannot leave $[-r,r]\times [0,1]$ nor $[0,1]\times [-r,r]$. Thus $E$ is trapped in $[-r,r]^2$ and does not evolve by the gradient flow of $\E$ which is given by circles of increasing radius from $E(0)$.
\end{proof}

On the whole space, we could use the previous results to show that the circle is in fact non-expanding since the angles $\phi$ with $\tan\phi\in \mathbb Q$ are dense in $(-\pi/2, \pi/2)$.

\section{Related Models}\label{section stuff}

Let us briefly discuss the validity of our results for similar models concerning the same phenomenon. The first two extensions we discuss concern the dissipation mechanism, while the third one discusses a modification of the energy functional.

\subsection{Non-viscous Evolution}

It can be argued that the use of a quadratic dissipation is unphysical for the dynamics of dislocations and a rate independent evolution
\begin{equation}\label{eq non-viscous}
-\delta \E_\eps(u_\eps) \in c_\eps \,\sign(\dot u_\eps), \qquad c_\eps>0
\end{equation}
associated to a linear dissipation would be physically more sensible. Here
\[
\sign(z) = \begin{cases}\{1\} &z>0\\
[-1,1] &z=0\\ 
\{-1\} &z<0\end{cases}
\]
is the usual set-valued sign function. We believe that the emergence of an asymmetric, stick-slip type motion law from a viscous dissipation is the more interesting observation, in particular as rate-independent dynamics of this problem appear to be stationary in many cases. Namely, the sub-solutions constructed above satisfy
\[
-c_\eps \leq \frac1{|\log\eps|}\left(\A u_\eps - \frac1\eps\,W'(u_\eps)\right) \leq c_\eps
\]
for $c_\eps \geq C\,\frac{\eps}{|\log\eps|\,d_\eps^2}$ for suitable $C>0$ in the case of Theorem \ref{theorem slowness line} and $c_\eps \geq C/|\log\eps|$ in the case of Theorems \ref{theorem step line} and \ref{theorem step circle}. The same is true for many similar initial conditions even without the sign condition. Thus for such a choice of $c_\eps$, the rate-independent evolution can be taken as stationary. 
The same holds for an evolution law associated to a mixed dissipation
\[
-\delta \E_\eps(u_\eps) \in c_\eps \,\sign(\dot u_\eps) + d_\eps\,\dot u_\eps
\]
if $c_\eps$ is chosen in the corresponding parameter regime as above, since the last term vanishes identically for stationary solutions of the differential inclusion \eqref{eq non-viscous}.

\subsection{Finite Relaxation Speed}

We considered the $L^2$-gradient flow of the energy
\[
\E(u) = \frac1{|\log\eps|} \left(\frac12\,[u]_{1/2}^2 + \frac1\eps \int_{\T^2}W(u)\dx\right)
\]
which arose from as an equilibrium localisation on a plane of the crystal grid of the energy
\[
\F_\eps(u) = \frac1{|\log\eps|} \left(\frac12\int_{\T^2\times \R_+}|\nabla u|^2\dx  + \frac1\eps \int_{\T^2}W(u)\dx\right).
\]
The modelling assumption behind this mechanism is that for given dislocations in a plane, the rest of the crystal has relaxed to the minimal Dirichlet energy, which is not quite true in the dynamic case. When we consider the first variation
\begin{align*}
\delta F_\eps(u; \phi) &= \frac1{|\log\eps|} \left(\int_{\T^2\times \R_+}\langle \nabla u, \nabla \phi\rangle \dx  + \frac1\eps \int_{\T^2}W'(u)\phi\dx\right)\\
	&= \frac1{|\log\eps|} \left(\int_{\T^2\times \R_+}\langle \nabla \cdot(\phi\,\nabla u) - (\Delta u)\phi \dx  + \frac1\eps \int_{\T^2}W'(u)\phi\dx\right)\\
	&= \frac1{|\log\eps|} \left(\int_{\T^2\times \R_+} - (\Delta u)\phi \dx  +  \int_{\T^2}\partial_\nu u + W'(u)\phi\dx\right)
\end{align*}
and the inner product
\[
\langle v,\phi\rangle := \frac1{m_\eps}\int_{\R_+\times \T^2}v\phi \dx + \frac1\eps \int_{\T^2}v\phi,
\]
we obtain an evolution equation
\begin{equation}\label{eq finite speed}
\begin{pde}
m_\eps\,u_t &= \frac1{|\log\eps|} \Delta u&\text{in }\T^2\times \R_+\\
\eps\,u_t &= \frac1{|\log\eps|}\left(\A u - \frac1\eps\,W'(u_\eps)\right) &\text{on }\T^2\setminus\bigcup_{k=1}^{N_\eps}\overline{B_\eps(x_{i,\eps})}\\
u &= 0 &\text{on }\T^2\setminus\bigcup_{k=1}^{N_\eps} \overline{B_\eps(x_{i,\eps})}.
\end{pde}
\end{equation}
The case we considered above corresponds to an infinitely fast relaxation speed in the half-space, i.e.\ the formal limit $m_\eps \equiv 0$. In that case we could forget about the analytic continuation and only had to track the evolution of the boundary values. We can connect this to the case of positive $m_\eps>0$ as follows:

All the sub-solutions, super-solutions and solutions constructed above for the gradient flow equation on $\Omega = \R^2$ or $\Omega = \T^2$ were pointwise non-increasing, so their harmonic extensions to $\Omega\times \R_+$  have this property as well. Thus the analytic continuation $\hat u_\eps$ satisfies
\[
m_\eps \partial_t \hat u_\eps \leq 0 = \frac1{|\log\eps|} \Delta \hat u_\eps \quad\text{ in }\Omega\times\R_+
\]
either in the viscosity or the distributional sense, which means that the analytic solution $\hat u_\eps$ is a sub-solution for \eqref{eq finite speed}. In this sense, we can at least say that an evolution with a finite relaxation speed can in no case be faster than the limiting case we considered.

\subsection{Finite-Strength Pinning}

The hard constraint $u\equiv 0$ on $\bigcup_{k=1}^{N_\eps} \overline{B_\eps(x_{i,\eps})}$ can be see as the limiting case of the following soft obstacle problems. We consider a version of $\E_\eps$ on the whole space with an additional term in the energy
\[
\F_\eps(u) = \frac1{|\log\eps}\left([u]_{1/2}^2 + \int_{\T^2}W(u)\dx\right) + \sum_{i=1}^{N_\eps}\frac1{\eps^2}\int_{B_\eps(x_{i,\eps})} g\left(\frac{x-x_{i,\eps}}\eps\right)\,|u|^{q}\dx.
\]
Here $1\leq q < \infty$ is a parameter we could choose freely and $g\in C_c(B_1(0))$ is a non-negative function. The hard obstacle arises as the formal limit $g\to +\infty\cdot\chi_{\bigcup_{k=1}^{N_\eps} \overline{B_\eps(x_{i,\eps})}}$ for any $q$. Physically, the case $q=1$ seems the most relevant. This extension has been discussed in \cite{MR2178227, MR2231783}, and the same $\Gamma$-limit statement still holds with a different capacity function $\alpha: \Z\to [0,\infty)$.

Our results apply also here by the following considerations. Observe that
\[
\ul u_t = 0, \quad W'(\ul u) = 0, \quad \A\ul u \geq 0 
\]
at $x\in \bigcup_{k=1}^{N_\eps} \overline{B_\eps(x_{i,\eps})}$, so $\ul u$ is a sub-solution also to $u_t = \A u - \frac 1\eps W'(u) - q\,g(x)\,|u|^{q-2}u$ which is the gradient flow equation of $\F_\eps$. Thus we can use the same sub-solutions to obtain upper bounds on the velocity of interfaces, even in the case $q=1$ since the sub-solutions $\ul u_\eps$ do not change sign. For super-solutions, we have to solve a minimisation problem with the soft pinning instead of the hard one instead:
\[
\text{Minimise } \frac12\,[u]_{1/2}^2 + \int_{S^1_l}W(u)\dx + \int_{B_1(x_0)}g\,|u|^q\dx \qquad\text{subject to }\frac1l\int_{S^1_l}u\dx >\frac12.
\]
When we establish that the solution to this problem satisfies $u\not\equiv 1$, $0\leq u\leq 1$, we obtain matching bounds on the scaling of the velocity of interfaces at a single step on the real line and up to a logarithmic factor in the plane. We observe that also here, the contracting effect of the obstacles vanishes compared to the kink/anti-kink attraction.

\section{Conclusion}\label{section summary}

We have identified the time-scale on which a pinning constraint would naturally act by considering a whole space problem (up to a factor of $O(|\log\eps|)$). We have shown that the gradient flows of the pinned problems do not converge to the gradient flow of the limiting problem under certain assumptions on the distribution of obstacles and given estimates on the behaviour for certain initial conditions. A number of questions remain open.

\begin{enumerate}
\item Is there an explicit law that describes the limit of the evolutions of the $\eps$-problem at curved initial conditions?

\item How dependent is the limiting motion on the exact (well-prepared) initial condition?

\item Do the same results hold for more general distributions of obstacles, or can other phenomena occur for less regularly distributed (or moving) obstacles? 
\end{enumerate}

Furthermore, our methods used the rotational symmetry of the fractional Laplacian and the fact that all functions we constructed were non-negative. We expect that these constraints could be eliminated. We also believe that a more explicit characterisation of admissible potentials $W$ should be available.

\appendix

\section{Fractional Evolution Equations}

The gradient flow of the energy $\E_\eps$ is given by the fractional parabolic equation
\begin{equation}\label{eq gradient flow}
\begin{pde}
c_\eps\eps\,u_t &= \frac1{|\log\eps|}\left(\A u - W'(u)\right) &t>0, x\in \T^2\setminus \bigcup_{i=1}^{N_\eps}\overline{B_\eps(x_{i,\eps})}\\
		u &\equiv 0& t\geq 0, x\in \bigcup_{i=1}^{N_\eps}\overline{B_\eps(x_{i,\eps})}\\
		u &= u^0 & t=0
\end{pde}
\end{equation}
with $c_\eps\equiv 1$ which formally has the structure
\begin{equation}\label{eq fractional pde}
\begin{pde}
u_t &= \A u + f(u) &t>0,\: x\in \Omega\\
		u &\equiv 0& t\geq 0,\: x\in \T^2\setminus \Omega\\
		u &= u^0 & t=0,\: x\in \Omega
\end{pde}
\end{equation}
where $\A = - (-\Delta)^{1/2}$ is the fractional Laplacian of order $s=1/2$ and $f$ is a bounded Lipschitz function. Choosing constants $c_\eps \ll 1$ corresponds to accelerating time to rescale slow motion of the gradient flows to the macroscopic time-scale. Since we derived the equation as the gradient flow of the energy $\E_\eps$, the most natural concept of a solution is that of a weak solution in a Bochner space
\[
W_\eps = \left\{u \in L^2\left([0,T], X_\eps\right)\bigg| \frac{\d u}{\dt} \in L^2\left([0,T], X_\eps^*\right)\right\}
\]
where again
\[
X_\eps:= \{u_\eps \in H^{1/2}(\T^2)\:\big|\:u_\eps\equiv 0\text{ on }B_\eps(x_{i,\eps})\text{ for } 1\leq i\leq N_\eps\}.
\]
 It is well-known that the operator $\A: H^{1/2}(\T^2)\to H^{-1/2}(\T^2)$ (and also $\A:X_\eps\to X_\eps^*$) is monotone. Furthermore, if $u$ solves \eqref{eq fractional pde}, we note that $v(t,\cdot) = e^{-\lambda t}$ solves the same equation with $f$ replaced by
 \[
 f_\lambda(v) = e^{-\lambda t}\,f\left(e^{\lambda t} v\right) + \lambda v
 \]
which can be made monotone increasing for large enough $\lambda$ as $f$ is Lipschitz. The existence of a weak solution $v$ follows by the theory of monotone operators. Reversing the modification, we obtain a weak solution $u = e^{\lambda t} v$ in the Bochner space $W_\eps$.

If $u$ is smooth, $0\leq u_0\leq 1$, the comparison principle implies that $u\leq 1$ for all times and the non-linearity $f = - W'$. We  show that remains true for weak solutions in the appendix (after a suitable modification of $f$ away from $[0,1]$, which can be neglected after showing that $0\leq u\leq 1$).

Since $f$ is bounded, we have $f(u)\in L^\infty(\Omega)$ and \cite{MR3626038} implies that $u$ is H\"older continuous up to the boundary (note that the concept of a weak solution used in that article is weaker than ours). Hence $f(u)$ is also H\"older continuous, which means that $u$ is locally $C^{1,\alpha}$-smooth in $\Omega$ and $C^{0,1/2}$-continuous on $\T^2$. While the results of \cite{MR3626038} are given in domains in $\R^n$, the proofs also apply in the periodic case. 

In particular, the weak solutions are classical and thus justify the usual calculations that imply a decrease of energy along the time evolution. The solutions are in particular viscosity solutions and we may construct viscosity sub- and super-solutions to understand their behaviour.

Interestingly, the regularity results apply more easily if $f$ is bounded a priori. However, if $W(z) = (z^2-1)^2$ is the usual double-well potential and the initial condition lies between $-1$ and $1$, we may modify $W$ outside $[-1,1]$ to become bounded Lipschitz. In a second step, we may apply the maximum principle to deduce that solutions remain between $\pm 1$, which means that the solution to the modified problem is actually also a solution to the original problem -- compare the proof of Lemma \ref{lemma obstacle sub-solution}.

We will also investigated solutions of the evolution equation on the whole real line or in the plane. By the same arguments as above, weak solutions exist if the initial condition happens to lie in the space
\[
\left\{u\in L^2(\R^n)\:\bigg|\:u \equiv 0 \qquad\text{on }\bigcup_{i=1}^\infty B_\eps(x_{i,\eps})\right\}.
\]
Here we use that we could modify $f$ to become monotone and use a monotone Nemickij operator rather than having to pass to the theory of pseudomonotone operators, where the lack of compactness in the embedding for the whole space problem causes additional challenges. 

When we consider solutions to the evolution equation \eqref{eq gradient flow} on the real line with initial conditions approximating a single step function, on the other hand, we need to understand solutions in the viscosity sense. On the whole space, the theory of viscosity solutions for fractional evolution equations is developed \cite{MR2121115,MR2259335,MR2422079} and the pinning constraint could be included in the proof of the maximum principle by the doubling of variables in the standard way -- see e.g. \cite[Theorem 2]{MR2121115}. Existence can then be proved using Perron's method.

Consider the Bochner space $W$ over $\Omega\subset\T^2$ and the particular case of a non-linearity $f$ which satisfies $f(1) = 0$, $f<0$ on $(1,\infty)$ and $f$ is constant close to $\infty$. Assume further that we have an initial condition $u_0\leq 1$. Now consider $u_+:= \max\{u,1\}$. Since $u\in C^0([0,T], L^2(\T^2))$ by embedding theorems $u_+(t,\cdot)$ is well-defined in $L^2(\T^2)$ and we can calculate the integral
\[
\left(\int_{\T^2}(u_+)^2\dx\right)(t)
\]
pointwise in time. Due to Bochner-space theory, smooth functions are dense in $W$ and we can consider a sequence of functions $u_n \in C^0([0,T)\times \T^2) \cap C^2((0,T)\times \Omega)$ such that $u_n\equiv 0$ on $\T^2\setminus \Omega$ such that $u_n\to u$ in $W$. In particular this convergence implies
\[
\left(\int_{\T^2}(u_n)_+^2\dx\right)(0)\to 0.
\]
Take $(t,x)$ such that $u_n(t,x)>1$. By continuity, $u_n>1$ in a neighbourhood of the point, and the Laplacian can be calculated pointwise as a singular integral
\[
\left[\A u_{n,+}\right](t,x) = \int_{\Omega} \frac{u_{n,+}(t,y) - u_{n,+}(t,x)}{|x-y|^{3}}\dy \geq \int_{\Omega} \frac{u_n(t,y) - u_n(t,x)}{|x-y|^{3}}\dy = \left[\A u_n\right](t,x)
\]
since $u_{n,+}(t,y)\geq u(t,y)$ and $u_{n,+}(t,x) = u_{n,+}(t,x)$. On the other hand, if $u_{n,+}(t,x)= 1$, then $u_{n,+}$ is minimal at $(t,x)$ and thus $\A u_{n,+}\geq 0$ at $(t,x)$ in the distributional sense. It follows that 
\[
\left(\partial_t - \A\right) u_{n,+} \geq \chi_{\{u_n>1\}}\cdot \left(\partial_t - \A\right) u_n.
\]
Since $u_+$ is smooth enough to be a Sobolev function, it also lies in $W$ and we compute
\begin{align*}
\left(\int_{\R^n}(u_n)_+^2\dx\right)(t) &= \left(\int_{\R^n}(u_n)_+^2\dx\right)(0) + 2\int_0^t\left\langle\left[\partial_t - \A + \A\right](u_n)_+, (u_n)_+\right\rangle_{X^*, X} \d s\\
	&\leq \left(\int_{\R^n}(u_n)_+^2\dx\right)(0) + 2\int_0^t\left\langle\left[\partial_t - \A\right](u_n)_+, (u_n)_+\right\rangle_{X^*, X} \d s\\
	&\leq \left(\int_{\R^n}(u_n)_+^2\dx\right)(0) + 2\int_0^t\left\langle f(u_{n,+}), (u_n)_+\right\rangle_{L^2, L^2} + \langle \eta_n, u_{n,+}\rangle_{X^*,X} \d s\\
	&\leq \left(\int_{\R^n}(u_n)_+^2\dx\right)(0) + 2\int_0^t \langle \eta_n, u_{n,+}\rangle_{X^*,X} \d s\\
	&\to 0
\end{align*}
since $u_n\to u$ strongly in $W$, thus in particular $f(u_n)\to f(u)$ strongly as well. This implies that $u\leq 1$ for all times and thus the solutions are classical for positive times. The same argument shows $u\geq 0$ and a slight modification implies comparison with a stationary sub- or super-solution.


\begin{thebibliography}{FRRO17}

\bibitem[AB02]{MR1907765}
N.~Ansini and A.~Braides.
\newblock Asymptotic analysis of periodically-perforated nonlinear media.
\newblock {\em J. Math. Pures Appl. (9)}, 81(5):439--451, 2002.

\bibitem[ABS98]{MR1657316}
G.~Alberti, G.~Bouchitt{{\'e}}, and P.~Seppecher.
\newblock Phase transition with the line-tension effect.
\newblock {\em Arch. Rational Mech. Anal.}, 144(1):1--46, 1998.

\bibitem[BI94]{baernstein1994unified}
A.~Baernstein~II.
\newblock A unified approach to symmetrization.
\newblock In {\em Partial Differential Equations of Elliptic Type, Eds. A.
  Alvino et al., Symposia Matematica}, volume~35, pages 47--91, 1994.

\bibitem[BI08]{MR2422079}
G.~Barles and C.~Imbert.
\newblock Second-order elliptic integro-differential equations: viscosity
  solutions' theory revisited.
\newblock {\em Ann. Inst. H. Poincar{\'e} Anal. Non Lin{\'e}aire},
  25(3):567--585, 2008.

\bibitem[BK90]{MR1075075}
L.~Bronsard and R.~V. Kohn.
\newblock On the slowness of phase boundary motion in one space dimension.
\newblock {\em Comm. Pure Appl. Math.}, 43(8):983--997, 1990.

\bibitem[CM97]{MR1493040}
D.~Cioranescu and F.~Murat.
\newblock A strange term coming from nowhere.
\newblock In {\em Topics in the mathematical modelling of composite materials},
  volume~31 of {\em Progr. Nonlinear Differential Equations Appl.}, pages
  45--93. Birkh{\"a}user Boston, Boston, MA, 1997.

\bibitem[CP89]{carr1989metastable}
J.~Carr and R.~L. Pego.
\newblock Metastable patterns in solutions of $u_t= \epsilon^2 u_{xx-} f (u)$.
\newblock {\em Communications on pure and applied mathematics}, 42(5):523--576,
  1989.

\bibitem[CSM05]{cabre2005layer}
X.~Cabr{\'e} and J.~Sol{\`a}-Morales.
\newblock Layer solutions in a half-space for boundary reactions.
\newblock {\em Communications on pure and applied mathematics},
  58(12):1678--1732, 2005.

\bibitem[DI06]{MR2259335}
J.~Droniou and C.~Imbert.
\newblock Fractal first-order partial differential equations.
\newblock {\em Arch. Ration. Mech. Anal.}, 182(2):299--331, 2006.

\bibitem[DKY08]{MR2463225}
N.~Dirr, G.~Karali, and N.~K. Yip.
\newblock Pulsating wave for mean curvature flow in inhomogeneous medium.
\newblock {\em European J. Appl. Math.}, 19(6):661--699, 2008.

\bibitem[DR10]{Demir:10}
E.~Demir and D.~Raabe.
\newblock Mechanical and microstructural single-crystal bauschinger effects:
  Observation of reversible plasticity in copper during bending.
\newblock {\em Acta Materialia}, 58(18):6055--6063, 2010.

\bibitem[DY06]{MR2231253}
N.~Dirr and N.~K. Yip.
\newblock Pinning and de-pinning phenomena in front propagation in
  heterogeneous media.
\newblock {\em Interfaces Free Bound.}, 8(1):79--109, 2006.

\bibitem[FRRO17]{MR3626038}
X.~Fern{\'a}ndez-Real and X.~Ros-Oton.
\newblock Regularity theory for general stable operators: parabolic equations.
\newblock {\em J. Funct. Anal.}, 272(10):4165--4221, 2017.

\bibitem[GM05]{MR2178227}
A.~Garroni and S.~M{{\"u}}ller.
\newblock {$\Gamma$}-limit of a phase-field model of dislocations.
\newblock {\em SIAM J. Math. Anal.}, 36(6):1943--1964 (electronic), 2005.

\bibitem[GM06]{MR2231783}
A.~Garroni and S.~M{{\"u}}ller.
\newblock A variational model for dislocations in the line tension limit.
\newblock {\em Arch. Ration. Mech. Anal.}, 181(3):535--578, 2006.

\bibitem[GM12]{gonzalez2010slow}
M.~d.~M. Gonzalez and R.~Monneau.
\newblock Slow motion of particle systems as a limit of a reaction-diffusion
  equation with half-{L}aplacian in dimension one.
\newblock {\em DCDS-A}, 32(4):1255--1286, 2012.

\bibitem[Imb05]{MR2121115}
C.~Imbert.
\newblock A non-local regularization of first order {H}amilton-{J}acobi
  equations.
\newblock {\em J. Differential Equations}, 211(1):218--246, 2005.

\bibitem[IS09]{imbert2009phase}
C.~Imbert and P.~Souganidis.
\newblock Phase field theory for fractional reaction-diffusion equations and
  applications.
\newblock {\em preprint arXiv:0907.5524}, 2009.

\bibitem[KCO02]{MR1935021}
M.~Koslowski, A.~M. Cuiti{\~n}o, and M.~Ortiz.
\newblock A phase-field theory of dislocation dynamics, strain hardening and
  hysteresis in ductile single crystals.
\newblock {\em J. Mech. Phys. Solids}, 50(12):2597--2635, 2002.

\bibitem[Kur06]{KurzkeBV}
M.~Kurzke.
\newblock Boundary vortices in thin magnetic films.
\newblock {\em Calc. Var. Partial Differential Equations}, 26(1):1--28, 2006.

\bibitem[Kur07]{KurzkeGF}
M.~Kurzke.
\newblock The gradient flow motion of boundary vortices.
\newblock {\em Ann. Inst. H. Poincar\'e Anal. Non Lin\'eaire}, 24(1):91--112,
  2007.

\bibitem[Mie12]{MR2992854}
A.~Mielke.
\newblock Emergence of rate-independent dissipation from viscous systems with
  wiggly energies.
\newblock {\em Contin. Mech. Thermodyn.}, 24(4-6):591--606, 2012.

\bibitem[MK74]{MR0601059}
V.~A. Marchenko and E.~Y. Khruslov.
\newblock {\em {\cyr {K}raevye zadachi v oblastyakh s melkozernisto\u{i} \
  granitse\u{i}}}.
\newblock Izdat. ``Naukova Dumka'', Kiev, 1974.

\bibitem[PV15]{MR3338445}
S.~Patrizi and E.~Valdinoci.
\newblock Crystal dislocations with different orientations and collisions.
\newblock {\em Arch. Ration. Mech. Anal.}, 217(1):231--261, 2015.

\bibitem[PV16]{Patrizi:2016aa}
S.~Patrizi and E.~Valdinoci.
\newblock Long-time behavior for crystal dislocation dynamics.
\newblock 09 2016.

\bibitem[ROS14]{ros2014dirichlet}
X.~Ros-Oton and J.~Serra.
\newblock The {D}irichlet problem for the fractional {L}aplacian: regularity up
  to the boundary.
\newblock {\em Journal de Math{\'e}matiques Pures et Appliqu{\'e}es},
  101(3):275--302, 2014.

\bibitem[Ser11]{serfaty2011gamma}
S.~Serfaty.
\newblock Gamma-convergence of gradient flows on {H}ilbert and metric spaces
  and applications.
\newblock {\em Discrete Contin. Dyn. Syst}, 31(4):1427--1451, 2011.

\bibitem[SS04]{sandier2004gamma}
E.~Sandier and S.~Serfaty.
\newblock Gamma-convergence of gradient flows with applications to
  {G}inzburg-{L}andau.
\newblock {\em Communications on Pure and Applied mathematics},
  57(12):1627--1672, 2004.

\bibitem[SV12]{MR2948285}
O.~Savin and E.~Valdinoci.
\newblock {$\Gamma$}-convergence for nonlocal phase transitions.
\newblock {\em Ann. Inst. H. Poincar{\'e} Anal. Non Lin{\'e}aire},
  29(4):479--500, 2012.

\bibitem[SV14]{MR3161511}
R.~Servadei and E.~Valdinoci.
\newblock Weak and viscosity solutions of the fractional {L}aplace equation.
\newblock {\em Publ. Mat.}, 58(1):133--154, 2014.

\end{thebibliography}
\end{document}